\documentclass[a4paper,11pt,bibliography=totoc,listof=totoc,headinclude=true,cleardoublepage=empty,oneside]{scrartcl}

\usepackage[english]{babel}
\usepackage[utf8]{inputenc}
\usepackage[inline]{enumitem}
\usepackage{fullpage}
\usepackage{ifthen}
\usepackage{color}
\usepackage{amsmath,amsthm,amssymb,amsfonts, tikz, wasysym}
\usepackage{graphicx}
\usepackage{psfrag}
\usepackage{algorithm,algorithmic}
\usepackage{tikz}
\usepackage{pgfplots}
\usepackage{float}
\usepackage[version=4]{mhchem}
\usepackage[labelfont=bf]{caption}
\usepackage{subfig}
\usepackage{placeins}
\usetikzlibrary{decorations.pathreplacing}

\numberwithin{equation}{section}

\makeatletter

\newcommand\biggg{\bBigg@{3.5}}
\newcommand*\bigcdot{\mathpalette\bigcdot@{.5}}
\newcommand*\bigcdot@[2]{\mathbin{\vcenter{\hbox{\scalebox{#2}{$\m@th#1\bullet$}}}}}
\newcommand{\revision}[1]{{\color{black} #1}}

\newcommand{\eremk}{\hbox{}\hfill\rule{0.8ex}{0.8ex}}

\makeatother

\newcommand{\Rn}[1]{\uppercase\expandafter{\romannumeral#1}}
\newcommand{\skp}[1]{\left< #1 \right>}

\theoremstyle{plain}
\newtheorem{theorem}{Theorem}[section]
\newtheorem{lemma}[theorem]{Lemma}
\newtheorem{corollary}[theorem]{Corollary}
\newtheorem{proposition}[theorem]{Proposition}
\newtheorem{remark}[theorem]{Remark}
\newtheorem{definition}[theorem]{Definition}
\newtheorem{example}[theorem]{Example}
\newtheorem{assumption}[theorem]{Assumption}

\usepackage[unicode,colorlinks=true,linkcolor=black,citecolor=black,urlcolor=black,pagebackref=false]{hyperref} 

\KOMAoptions{footinclude=false} 
\KOMAoptions{headsepline=true} 
\KOMAoptions{DIV=12} 
\recalctypearea 

\AfterCalculatingTypearea{%
	\addtolength{\oddsidemargin}{0.5cm}
	\addtolength{\evensidemargin}{-0.5cm}
}
\recalctypearea

\title{An implementation of $hp$-FEM for the fractional Laplacian}

\author{Björn Bahr\footnote{
Institute of Analysis and Scientific Computing, TU Wien, Vienna, Austria,
\texttt{bjoern.bahr@tuwien.ac.at}}, \,
Markus Faustmann\footnote{
Institute of Analysis and Scientific Computing, TU Wien, Vienna, Austria,
\texttt{markus.faustmann@tuwien.ac.at}}, 
\, and Jens Markus Melenk\footnote{
Institute of Analysis and Scientific Computing, TU Wien, Vienna, Austria,
\texttt{melenk@tuwien.ac.at}}}
\date{\today}

\begin{document}
\selectlanguage{english}
\maketitle
\begin{abstract}
We consider the discretization of the $1d$-integral Dirichlet fractional Laplacian by $hp$-finite elements. We present quadrature schemes to set up the stiffness matrix and load vector that preserve the exponential convergence of $hp$-FEM on geometric meshes. The schemes \revision{are} based on Gauss-Jacobi and Gauss-Legendre rules. We show that taking a 
number of quadrature points slightly exceeding the polynomial degree is enough to preserve root exponential convergence. The total number of \revision{algebraic operations} 
to set up the system is $\mathcal{O}(N^{5/2})$, where $N$ is the problem size. \revision{Numerical example illustrate the analysis. We also extend our analysis to the fractional Laplacian in higher dimensions for $hp$-finite element spaces based on shape regular meshes.} 
\end{abstract}

\section{Introduction}
Fractional differential equations have become an important modelling tool, which sparked significant research in analysis and design and analysis of numerical methods, see, e.g., \cite{bucur2016applications} and, for numerical methods,  
\cite{bonito2018overview,borthagaray2019overview,lischke2020overview,delia20overview,jin-zhou23,handbook_vol_3,zayernouri-wang-shen-karniadakis23} 
and references therein.

We consider the fractional differential equation
\begin{subequations}\label{eq:fractionalPDE}
\begin{align}
	(-\Delta)^{s}u=f \qquad &\text{in } \Omega:=(-1,1) \subset {\mathbb R}, \\
	u=0 ~~~~&\text{in } \Omega^{c}:=\mathbb{R}\setminus \overline{\Omega},
\end{align}
\end{subequations}
where $s \in (0,1)$, and $f $ is analytic in $\overline{\Omega}$. Here, the operator $(-\Delta)^{s}$ is the Dirichlet integral fractional Laplacian, defined in (\ref{eq:def-FL}) below. 
Among the discretization techniques, methods like the $hp$-finite element method (FEM) stand out as they achieve exponential convergence, \cite{bahr2023exponential,faustmann2023hp}, so that significantly  fewer degrees of freedom \revision{are required to achieve the same accuracy} compared to fixed order methods such as the classical $h$-FEM. 
This is particularly interesting for non-local problems such as fractional PDEs since there the stiffness matrices are fully populated with corresponding \revision{high} memory requirements and high complexity to set up the \revision{matrices}.  

In fact, \cite{bahr2023exponential} considers $hp$-FEM approximations on suitably designed geometric meshes in one space dimension and shows, 
for the $hp$-FEM approximation $u_N$ to the solution $u$ of \eqref{eq:fractionalPDE},
the energy-norm error estimate 
\begin{align}\label{eq:exponential_convergence}
\|u-u_{N}\|_{\widetilde{H}^{s}(\Omega)} \le C \exp(-b\sqrt{N}),
\end{align}
where \revision{$b,C > 0$} are constants independent of the problem size $N$. \revision{Such exponential convergence results generalize to higher dimensions, e.g., 
in two space dimensions \cite{faustmann2023hp} asserts a similar convergence estimate where the square root in the exponent is replaced by $N^{1/4}$.}

The exponential convergence in \cite{bahr2023exponential,faustmann2023hp} is asserted ignoring variational crimes, in particular, it is shown 
under the assumption that $u_N$ is the exact $hp$-finite element Galerkin approximation to $u$. 
However, a practical realization of the Galerkin method (\ref{weak_formulation}) requires the evaluation of singular
integrals by numerical quadrature. In the present work we develop and analyze quadrature schemes that preserve the exponential
convergence \eqref{eq:exponential_convergence}. The quadratures are based on Gauss-Legendre and Gauss-Jacobi \revision{rules}, and the analysis is performed 
in the framework of the First Strang Lemma.
\revision{The key observation is that the hyper-singular integrand can be transformed such that singularities are aligned with coordinate axes, which allows for efficient treatment with Gauss-Jacobi rules.} 

\revision{The issue of evaluating singular integrals has already appeared in the context of boundary element methods (BEMs), \cite{sauter2011book}. For the kernels of \revision{BEM-operators} arising from 
second order elliptic boundary value problems, regularizing transformations for the singular integrals have been devised that fully remove the singularity so that standard quadrature techniques can be brought to 
bear and a full quadrature error analysis is available, \cite[Chap.~5]{sauter2011book}. For certain meshes with structure even the high order stiffness matrices of $hp$-BEM can be computed explicitly, \cite{maischak95, holm-maischak-stephan96}. 
 }
\revision{Generalizing the quadrature techniques described in \cite[Chap.~5]{sauter2011book} the works \cite{chernov-vonpetersdorff-schwab11,chernov-schwab12,chernov-reinarz13,chernov-vonpetersdorff-schwab15} present and analyze regularizing transformations 
for a class of integrands that includes products of analytic/Gevrey-regular functions and singular functions; computationally, an essential point of these transformations is that  they lead to the use of products of 
Gauss-Legendre and $hp$-quadrature or Gauss-Jacobi quadrature. Using similar transformations (in 1d) and building on these works (for $d > 1$), our analysis considers the specific case of $hp$-FEM for the fractional Laplacian, explicitly works out the dependence on the polynomial degree $p$ of the ansatz space, and asserts exponential convergence of the fully discrete method. The work to set up the stiffness matrix is algebraic in the problem size.} 

\revision{Implementations of the spectral fractional Laplacian have been proposed in the literature. Low order (for $d \ge 1$) Galerkin methods include 
\cite{acosta-bersetche-borthagaray17,ainsworth-glusa18,bebendorf-feist23} and typically exploit that a specific choice of basis is made in contrast to the 
present quadrature-based approach.
Especially for $1d$ fractional differential equations, spectral and spectral element methods are available in the literature, 
see, e.g., \cite{handbook_vol_3, lischke-handbook,shen-handbook,zayernouri-karniadakis13, zayernouri-wang-shen-karniadakis23,lischke2020overview,mao-shen18,chen-shen-wang16} 
and references therein. The $1d$ quadrature techniques employed in the present work on shape regular meshes are closely related to those presented independently
in \cite{mao-shen18}.
Compared to these works, an important novel aspect of the present work is the full quadrature error analysis that rigorously establishes that taking 
$n \ge p+1$ quadrature points 
($p>0$ denoting the employed polynomial degree) is sufficient to retain the exponential convergence of $hp$-FEM.}

\revision{In the present article, we consider the $1d$ case in great detail to make key concepts appear clearly. Extensions to $d>1$ are possible, but come with additional (technical) difficulties. 
We present an analysis for $d > 1$ for shape regular meshes based on the regularizing transformations of \cite{chernov-schwab12} in Section~\ref{sec:2D}. 
We hasten to add that exponential convergence (both in terms of error versus number of degrees of freedom and error versus computational work) of $hp$-FEM in $d \geq 2$ requires anisotropic elements with large aspect ratio, \cite{faustmann2023hp}. 
A quadrature error analysis for meshes including anisotropic elements is the topic of a forthcoming work.}

The present article is structured as follows: In Section~\ref{ch:main_result}, we introduce our model problem and formulate the main result, exponential convergence of $hp$-FEM in the presence of quadrature, in Theorem~\ref{main_result}. Section~\ref{ch:quadrature} specifies the Gaussian quadrature rules and the resulting approximation of the bilinear and linear forms in the weak formulation of the model problem. \revision{Section~\ref{sec:stability} shows stability of the method under quadrature. }
Section~\ref{ch:proof} provides the proofs of our main results using the First Strang Lemma, while the consistency analysis 
is postponed to Section~\ref{ch:consistency_errors}. \revision{Section~\ref{sec:2D} extends the $1d$-analysis to higher dimensions for shape regular meshes based
on the quadrature techniques developed in \cite{chernov-vonpetersdorff-schwab11,chernov-schwab12}. }

Finally, Section~\ref{ch:numerics} provides numerical examples illustrating
the \revision{performance} of the quadrature \revision{scheme}. 

\section{Main Results}\label{ch:main_result}
For $s\in (0,1)$, we consider the integral fractional Laplacian defined for univariate functions $u$ 
pointwise as the principal value singular integral  
\begin{align}
\label{eq:def-FL}
(-\Delta)^su(x) 
:= 
C(s) \; \text{P.V.} \int_{{\mathbb R}}\frac{u(x)-u(y)}{|x-y|^{1+2s}} \, dy 
\quad \text{with} \quad
C(s):= - 2^{2s}\frac{\Gamma(s+1/2)}{\pi^{1/2}\Gamma(-s)},
\end{align}
where $\Gamma(\cdot)$ denotes the Gamma function. 

Appropriate function spaces for fractional differential equations are fractional Sobolev spaces, defined for $t \in (0,1)$ and any open set \revision{$\omega \subset {\mathbb R}^d$} 
by means of the Aronstein-Slobodeckij seminorm 
\begin{align*}
|v|^2_{H^t(\omega)} 
= 
\int_{\revision{\omega}} \int_{\omega} \frac{|v(x) - v(y)|^2}{|x-y|^{\revision{d+2t}}}
\,dy\,dx, 
\qquad 
\|v\|^2_{H^t(\omega)} = \|v\|^2_{L^2(\omega)} + |v|^2_{H^t(\omega)}.
\end{align*}
In order to incorporate the exterior Dirichlet condition, we define $r(x):=\operatorname{dist}(x,\partial\Omega)$ 
and introduce the space 
$\widetilde{H}^{t}(\Omega) := \left\{u \in H^t(\revision{{\mathbb R}^d}) : u\equiv 0 \; \text{on} \; \revision{{\mathbb R}^d} \backslash \overline{\Omega} \right\}$ 
with norm 
\begin{align*}
\;  \|v\|_{\widetilde{H}^{t}(\Omega)}^2 := \|v\|_{H^t(\Omega)}^2 + \|v/r^t\|_{L^2(\Omega)}^2.
\end{align*}
\revision{With the exception of Section~\ref{sec:2D} the domain  $\Omega  = (-1,1)$ always denotes the bounded open interval from our model problem \eqref{eq:fractionalPDE};
in Section~\ref{sec:2D}, we will consider polyhedral $\Omega \subset \mathbb{R}^d$.} 
We will use the fact that the norm $\|\cdot\|_{\widetilde{H}^s(\Omega)}$ and the seminorm $|\cdot|_{H^s({\mathbb R})}$ are equivalent
on $\widetilde{H}^s(\Omega)$, \cite{mclean00}. 
%
The weak form of the fractional PDE \eqref{eq:fractionalPDE} reads: find $u \in \widetilde{H}^{s}(\Omega)$ such that
\begin{align}\label{weak_formulation}
	a(u,v):=\dfrac{C(s)}{2}\int_{\mathbb{R}}\int_{\mathbb{R}}\dfrac{(u(x)-u(y))(v(x)-v(y))}{|x-y|^{1+2s}} \,dy\,dx = \skp{f,v}_{L^{2}(\Omega)} =: l(v)
\end{align}
for all $v \in \widetilde{H}^{s}(\Omega)$. Since $a(\cdot,\cdot):\widetilde{H}^{s}(\Omega) \times \widetilde{H}^{s}(\Omega) \rightarrow \mathbb{R}$ 
is continuous and coercive on $\widetilde{H}^s(\Omega)$, (\ref{weak_formulation}) is uniquely solvable by the Lax-Milgram Lemma, 
see \cite[Sec.~{2.1}]{acosta2017femfractional}.

For the discretization of the weak formulation, we employ piecewise polynomials on shape regular meshes.
\begin{definition} [Shape regular meshes and spline spaces]
\label{def:meshes}

	For an interval $\mathcal{I} \subset \mathbb{R}$, we denote its length by $h_{\mathcal{I}}:=\text{diam}(\mathcal{I})$. 
For a bounded interval $\Omega = (x_0,x_M)$ let the points $x_0 < x_1 < \cdots < x_M$ determine the mesh ${\mathcal T}_\gamma = \{T_i:= (x_{i-1},x_i)\colon i=1,\ldots M\} $. 
        The mesh ${\mathcal T}_{\gamma}$ is said to be $\gamma$-shape regular, if 
	\begin{align}
		\gamma \: h_{T_{i}} \le h_{T_{j}} ~~~\text{for all}~~~ T_{i},T_{j} \in \mathcal{T}_{\gamma} ~~\text{with}~~  \overline{T_{i}} \cap \overline{T_{j}} \neq \emptyset.
	\end{align}

	Based on $ \mathcal{T}_{\gamma}$, we define finite dimensional spline spaces by 
	\begin{align*}
		&S^{p,1}(\mathcal{T}_{\gamma}):= \{ u \in H^{1}(\Omega) : u|_{T} \in \mathcal{P}_{p}(T) \text{ for all } T \in \mathcal{T}_{\gamma} \},\\
		&S_{0}^{p,1}(\mathcal{T}_{\gamma}):= S^{p,1}(\mathcal{T}_{\gamma}) \cap H_{0}^{1}(\Omega).
	\end{align*}
	Here, $\mathcal{P}_{p}(T)$ denotes the space of all polynomials with maximal degree $p \in \mathbb{N}$ on $T$.

The standard basis for $S^{p,1}_0({\mathcal T}_\gamma)$ is given by 
	\begin{align}\label{Basis}
		\mathcal{B} = \mathcal{B}^{lin} \cup \mathcal{B}^{Leg},
	\end{align}
	where $\mathcal{B}^{lin} := \{ \varphi_{i} : i=1,\dots,M-1 \} $ are the hat functions \revision{associated} with the interior nodes $x_{i},~i=1,\dots, M-1$ 
and $\mathcal{B}^{Leg}:=\cup_{T \in \mathcal{T}_{\gamma}} \mathcal{B}_{T}$ 
 with \emph{element bubble functions} ${\mathcal{B}_{T}}:= \{ \varphi_{T,i} : i=2,\dots, p \}$. For an element $T = (x_\ell,x_r)$ with length $h_{T} = x_r - x_\ell$, the element bubble functions are given by  
	\begin{align}
		\varphi_{T,i}(x)=\begin{cases}
			\int_{-1}^{-1+2(x-x_\ell)/h_{T}} P_{i-1}(t) ~dt & x \in T, \\
			0 & x \in \Omega \setminus \overline{T},
		\end{cases}
	\end{align}
where $P_i$ is the $i$-th Legendre polynomial.
\end{definition}
The $hp$-FEM approximation $u_{N}$ is \revision{given by} Galerkin discretization of (\ref{weak_formulation}):
Find $u_{N} \in S_{0}^{p,1}(\mathcal{T}_{\gamma})$ such that
\begin{align}\label{galerkin_approx}
	a(u_{N},v_{N})=l(v_{N}) ~~~\text{for all } v_{N} \in S_{0}^{p,1}(\mathcal{T}_{\gamma}).
\end{align}
\revision{
For a given basis $\mathcal{B}:=\{ \varphi_{1}, \dots, \revision{\varphi_{N}} \}$ of $S_{0}^{p,1}(\mathcal{T}_{\gamma})$, finding the solution $u_{N}:=\sum_{i=1}^N x_{i} \varphi_{i}$ is equivalent to setting up and solving the linear system
\begin{align}\label{galerkin_approx_linear_system}
	Ax=b,
\end{align}
\revision{where $A \in \mathbb{R}^{N\times N}$ with $A_{ij}=a(\varphi_{j},\varphi_{i})$ and $b \in \mathbb{R}^N$ with $b_{i}:=\skp{f,\varphi_{i}}_{L^{2}(\Omega)}$. Setting up the linear system requires evaluating the bilinear form $a(\cdot,\cdot)$ for all pairs of basis functions, which means calculating (singular) double integrals. Computing the linear form $l(\cdot)$ for all basis functions leads to a routine problem of calculating integrals involving $f$.}}

Our main convergence results are formulated for a specific kind of shape regular meshes, so-called geometric meshes, defined in the following \revision{Definition}~\ref{geometric_mesh_definition}. However, we \revision{emphasize} that the analysis of the consistency errors of the bilinear and linear forms in Chapter~\ref{ch:consistency_errors} hold for arbitrary shape regular meshes.
\begin{definition}
[Geometric mesh ${\mathcal T}^{L}_{geo,\sigma}$ and basis ${\mathcal B}^{geo}$ of the spline space $\revision{S^{p,1}_0}({\mathcal T}^{L}_{geo,\sigma})$]
\label{geometric_mesh_definition}
	Given a grading factor $ \sigma \in (0,1) $ and a number $ L \in \mathbb{N} $ of layers, the \emph{geometric mesh} $ \mathcal{T}^{L}_{geo,\sigma} = \{ T_{i}: i=1,...,2L+2 \} $ with $ 2L+2 $ elements $ T_{i}=(x^{geo}_{i-1},x^{geo}_{i}) $ is defined by the nodes
	\begin{align*}
		&x^{geo}_{0}:=-1,~~~ x^{geo}_{i}=-1+\sigma^{L-i+1} \text{ for } i=1,\dots,L, \\
		&x^{geo}_{i+1}=1-\sigma^{i-L} \text{ for } i=L,\dots,2L,~~~ x^{geo}_{2L+2} := 1.
	\end{align*}
We note that $N:=\dim S_{0}^{p,1}(\mathcal{T}^{L}_{geo,\sigma}) \sim pL$ and that $ \mathcal{T}^{L}_{geo,\sigma} $ is shape regular with $\gamma = \sigma$.  

The basis $\mathcal{B}^{geo}$ for $S_{0}^{p,1}(\mathcal{T}^{L}_{geo,\sigma})$ is taken as the basis of Definition~\ref{def:meshes} for the mesh
${\mathcal T}^{L}_{geo,\sigma}$. 
\end{definition}

    In \cite{bahr2023exponential} the following exponential convergence results for the energy norm error between the solution $u$ in \eqref{weak_formulation} 
and its $hp$-FEM approximation $u_N$ from \eqref{galerkin_approx} on geometric meshes $\mathcal{T}_{\gamma}=\mathcal{T}^{L}_{geo,\sigma}$ is shown: 
    \begin{proposition}[\cite{bahr2023exponential}]\label{theorem_galerkin}
Let $\mathcal{T}^{L}_{geo,\sigma}$ be a geometric mesh on the interval $ \Omega := (-1,1) $ with grading factor $ \sigma \in (0,1) $ and $ L $ layers of refinement towards the boundary points.   Let the data $ f $ be analytic in $ \overline{\Omega} $.  Let $u_{N} \in S_{0}^{p,1}(\mathcal{T}^{L}_{geo,\sigma})$ solve (\ref{galerkin_approx}) with $\mathcal{T}_{\gamma}=\mathcal{T}^{L}_{geo,\sigma}$ and $u$ solve (\ref{weak_formulation}). Then, there are $b$, $C > 0$ and for all $\varepsilon > 0 $ there 
is $C_\varepsilon > 0$ such that for all $ p $ and $ L $ there holds
    	\begin{align}
    		\|u-u_{N}\|_{\widetilde{H}^{s}(\Omega)} \le C e^{-bp} + C_\varepsilon \sigma^{(1/2-\varepsilon)L}.  
    	\end{align}

    		The choice $L \sim p$ leads to convergence $\|u-u_{N}\|_{\widetilde{H}^{s}(\Omega)} \le C \exp(-b'\sqrt{N})$, where $N$ is the dimension of $S_{0}^{p,1}(\mathcal{T}^{L}_{geo, \sigma})$ and $C,~b'$ are constants independent of $N$.
    \end{proposition}
    In practice, it is not possible to set up the linear system of equations corresponding to (\ref{galerkin_approx}) exactly due to the presence of the kernel
function $|x - y|^{-1-2s}$. To implement the $hp$-FEM method, we therefore have to work with computable numerical approximations 
$\widetilde{a}_{n}(\cdot,\cdot)$ and $\widetilde{l}_{n}(\cdot)$ of the bilinear form $a(\cdot,\cdot)$ and the right-hand side $l(\cdot)$, respectively. 
The fully discrete problem then reads:
 Find $\widetilde{u}_{N,n} \in S_{0}^{p,1}(\mathcal{T}_{\gamma})$ such that
\begin{align}\label{eq:fully_discrete}
	\widetilde{a}_{n}(\widetilde{u}_{N,n},v_{N})=\widetilde{l}_{n}(v_{N}) ~~~\text{for all } v_{N} \in S_{0}^{p,1}(\mathcal{T}_{\gamma}).
\end{align}
In Section~\ref{ch:quadrature} below, we specify the approximations $\widetilde{a}_{n}(\cdot,\cdot)$ and $\widetilde{l}_{n}(\cdot)$ based on (weighted) Gaussian quadrature rules with $n$ points.
Our main result formulated in the following states that the exponential convergence rate of $\widetilde{u}_{N}$ to the solution $u$ is still preserved.

\revision{
    \begin{theorem}[{Exponential convergence including quadrature}]\label{main_result}
 Let $\mathcal{T}^{L}_{geo,\sigma}$ be a geometric mesh on the interval $ \Omega := (-1,1) $ with grading factor $ \sigma \in (0,1) $ and $ L $ layers of refinement towards the boundary points. Let $f$ be analytic in $\overline{\Omega}$, denote by $u \in \widetilde{H}^s(\Omega)$ the solution to \eqref{weak_formulation} and by $\widetilde{u}_{N,n} \in S_{0}^{p,1}(\mathcal{T}^{L}_{geo,\sigma})$ the solution to \eqref{eq:fully_discrete} with $\mathcal{T}_{\gamma}=\mathcal{T}^{L}_{geo,\sigma}$, where $ \widetilde{a}_{n}(\cdot,\cdot) $ and $ \widetilde{l}_{n}(\cdot) $ are defined in (\ref{bilinear_aprox_definition}) and (\ref{functional_aprox_definition}), respectively. The index $n$ indicates the number of quadrature points that is used per integral and element.

\revision{There are constants $C$, $b > 0$ and, for each $\varepsilon > 0$ a constant $C_\varepsilon$ (\revision{depending on $f$, $s$,} and $\sigma$) such that for any $n \ge p+1$, $p$, $L \in {\mathbb N}$ there holds}
    	\begin{align}\label{ungl_main_result}
    		\|u-\widetilde{u}_{N,n}\|_{\widetilde{H}^{s}(\Omega)} &\le C e^{-br} + C_\varepsilon \sigma^{(1/2-\varepsilon)L} + C L^{2} r^{3} p^{3} \rho^{1+p+r-2n}. 
    	\end{align}
For $L \sim p$ and $n \ge p+1$ there holds 
in terms of the problem size $N := \dim S_{0}^{p,1}(\mathcal{T}^{L}_{geo, \sigma})$ for some $C,~b' > 0$ independent of $N$ and $n$  
    	\begin{align}\label{ungl_main_result-DOF}
\|u-\revision{\widetilde{u}_{N,n}}\|_{\widetilde{H}^{s}(\Omega)} \le C \exp(-b'\sqrt{N}).
\end{align}
For $L \sim p \sim n$ and the basis $\mathcal{B}^{geo} $ from Definition~\ref{geometric_mesh_definition},
the number of algebraic operations to set up the linear system corresponding to (\ref{galerkin_approx_linear_system}) is $\mathcal{O}(L^{5})  =
\mathcal{O}(N^{5/2})$. 
\end{theorem}
}

%
%
%
%
%

\section{Quadrature approximations}\label{ch:quadrature}
Throughout this section, we consider $\gamma$-shape regular meshes ${\mathcal T}_{\gamma}$. 
We start with some general definitions and notations. 
$\widehat{T}:=(0,1)$ \revision{denotes} the reference element and, for each element $T:=(x_{\ell},x_{r}) \in \mathcal{T}_{\gamma}$, we define the affine element map  by 
\begin{align}
	F_{T}: \widehat{T} \rightarrow T, ~~~ x \mapsto x_{\ell} + x \, h_{T}.
\end{align}
\revision{With a slight abuse of notation, we will naturally extend $F_T$ to an affine function ${\mathbb C} \rightarrow {\mathbb C}$ when needed.}
For a function $v$ defined on $T$, we write $\hat{v}_{T}$ for its pullback to the reference element
\begin{align}
	\hat{v}_{T} := v \circ F_{T}.
\end{align}
Our approximations to $a(\cdot,\cdot)$ and $l(\cdot)$ are based on the following (weighted) Gaussian quadrature rules.

Let $\omega: (0,1) \rightarrow \mathbb{R}$ be a positive, integrable weight function. Then, we approximate
\begin{align*}
I(\Phi):=\int\limits_{0}^{1}\Phi(x)\omega(x)~ dx \approx \sum_{i=1}^n \omega_i  \Phi(\xi_i) =: G_n(\Phi),
\end{align*}
where $\xi_i$ are the Gaussian quadrature nodes (zeros of orthogonal polynomials w.r.t.\ the $\omega$-weighted $L^2$-inner product) 
and $\omega_i = \int_0^1\omega(x) L_i(x) dx$ with the $i$-th Lagrange interpolation polynomials $L_i(x) = \prod_{j=1, j \ne i}^n \frac{x - \xi_j}{\xi_i - \xi_j} $
associated with the quadrature nodes $\xi_1,\ldots,\xi_n$. 
%

For $\omega \equiv 1$, we write $GL_n(\Phi)$ (Gauss-Legendre quadrature) for the quadrature rule. For integrands with singularities at the boundaries, we take  $ \omega(x)=(1-x)^{\alpha}x^{\beta}, ~~\alpha,\beta > -1 $ and write $GJ^{\alpha, \beta}_n(\Phi)$ (Gauss-Jacobi quadrature). For multivariate functions $\Phi(x,y)$, we will indicate 
by \revision{the subscript} $x$, $y$ the variable to which the quadrature rule is applied. 
\bigskip 

We start by deriving an approximation to the right-hand side $l(v):= \skp{f,v}_{L^{2}(\Omega)}$ in (\ref{galerkin_approx}). Dividing the integration domain $\Omega$ into the elements $T \in \mathcal{T}_{\gamma}$, transforming them to the reference element $\widehat{T}$, and using Gauss-Legendre quadrature for each integral defines the linear form $ \widetilde{l}_{n}(v_{N}) $ for $v_{N} \in S_{0}^{p,1}(\mathcal{T}_{\gamma})$ by 
\begin{align}\label{functional_aprox_definition}
	l(v_{N})=\int_{\Omega}v_{N}(x)f(x)~dx &= \sum_{T \in \mathcal{T}_{\gamma}} h_{T} \int\limits_{\widehat{T}}\hat{v}_{N,T}(x)\hat{f}_{T}(x)~dx \notag \\
	&\approx \sum_{T \in \mathcal{T}_{\gamma}} h_{T}~ GL_{n}(\hat{v}_{N,T}(x)\hat{f}_{T}(x)) =:\widetilde{l}_{n}(v_{N}).
\end{align}
The approximation of the bilinear form $a(\cdot,\cdot)$ is more involved since we have to deal with hyper-singular double integrals. Using symmetry and dividing the integration domain $\mathbb{R}\times\mathbb{R}$ into the elements $T \in \mathcal{T}_{\gamma}$ and the complementary set $\Omega^{c}$ leads to
\begin{align}
	a(v_{N},w_{N})= \frac{C(s)}{2} \Big( \sum\limits_{T \in \mathcal{T}_{\gamma}} \sum\limits_{T' \in \mathcal{T}_{\gamma}} I_{T,T'}(v_{N},w_{N}) + 2  \sum\limits_{T \in \mathcal{T}_{\gamma}} I_{T,\Omega^{c}}(v_{N},w_{N}) \Big),
\end{align}
where, for arbitrary sets $A,B \subset \mathbb{R}$, the symbol $I_{A,B}(v_{N},w_{N})$ denotes
\begin{align}
	I_{A,B}(v_{N},w_{N}) := \int\limits_{A} \int\limits_{B} \dfrac{(v_{N}(x)-v_{N}(y))(w_{N}(x)-w_{N}(y))}{|x-y|^{1+2s}} \,dy\,dx.
\end{align}
The integral over $\Omega^{c}$ can be integrated explicitly. All the other integrals have to be transformed to a reference square and then approximated by a suitable quadrature rule, which leads to four cases.

	\subsubsection*{Identical elements ($\boldsymbol{T=T'}$):}
	We transform the double integral $I_{T,T}(u,v)$ to the reference square $\widehat{T} \times \widehat{T} $ and divide this integration domain into the triangles $ A_{1} := \{ (x,y)~|~0 < x < 1, ~0 <y<x  \} $~ and $ A_{2} := \{ (x,y)~|~0 < y < 1, ~0 <x<y  \} $. As the integrand is invariant under the transformation $(x,y)\mapsto (y,x)$, we notice that both integrals are the same. Employing the Duffy transformation, i.e., $(x,y) \mapsto (x,xy)$, leads to
	\begin{align}\label{case_idpanel}
		I_{T,T}(v_{N},w_{N}) &=  2h_{T}^{1-2s} \int\limits_{\widehat{T}} \int\limits_{\widehat{T}} \dfrac{(\hat{v}_{N,T}(x)-\hat{v}_{N,T}(xy)) (\hat{w}_{N,T}(x)-\hat{w}_{N,T}(xy))}{|x-xy|^{2}} x^{2-2s}(1-y)^{1-2s} dydx \nonumber \\
&\approx  2h_{T}^{1-2s} GJ_{n,x}^{0,2-2s}\circ GJ_{n,y}^{1-2s,0} \left( \dfrac{(\hat{v}_{N,T}(x)-\hat{v}_{N,T}(xy)) (\hat{w}_{N,T}(x)-\hat{w}_{N,T}(xy))}{|x-xy|^2} \right) \nonumber \\ 
&=:  Q^n_{T,T}(v_{N},w_{N}).
	\end{align}
\revision{We note that after the separation of the weight function, the integrand in \eqref{case_idpanel} is a polynomial since only removable singularities are left.} 
\begin{remark}
\revision{Our choice of the Gauss-Jacobi weight function is not the only possible option, as, e.g., one could cancel out one power of $x$ in the first equality in \eqref{case_idpanel}. However, our choice is optimal in the sense that} it decreases the polynomial degree of the integrand \revision{as much as possible}.
\eremk
\end{remark}

	\subsubsection*{Adjacent elements ($\boldsymbol{T\ne T'}$ with $\boldsymbol{\overline{T}\cap \overline{T'}\neq \emptyset}$):}
Without loss of generality, we may assume that $T$ is the left neighbor of $T'$, otherwise $T$ and $T'$ change their roles. Then, the element maps transform the singularity at $\overline{T}\cap \overline{T'}$ to the point$(1,0)$ in the reference square. With an additional transformation $(x,y) \mapsto (1-x,y)$ we are now in a similar setting as in the previous case. The integral can be split into integrals over $A_1$ and $A_2$ and employing the Duffy transformation on $A_1$ (for $A_2$ we take $(x,y) \mapsto (xy,y)$) leads to
\begin{align}\label{case_neighbored}
	I_{T,T'}&(v_{N},w_{N}) =  h_{T}h_{T'} \bigg( \int\limits_{\widehat{T}} \int\limits_{\widehat{T}} \dfrac{(\hat{v}_{N,T}(1-x)-\hat{v}_{N,T'}(xy)) (\hat{w}_{N,T}(1-x)-\hat{w}_{N,T'}(xy))}{|h_{T}+yh_{T'}|^{1+2s}\,x^2} x^{2-2s} ~dydx\notag \\
	& + \int\limits_{\widehat{T}} \int\limits_{\widehat{T}} \dfrac{(\hat{v}_{N,T}(1-xy)-\hat{v}_{N,T'}(y)) (\hat{w}_{N,T}(1-xy)-\hat{w}_{N,T'}(y))}{|xh_{T}+h_{T'}|^{1+2s}\,y^2} y^{2-2s} ~dydx \bigg).
\end{align}
The singularities appear only in one variable in each integral, for which we employ Gauss-Jacobi quadrature, while in the other variable Gauss-Legendre quadrature is sufficient. This gives the approximation
\begin{align}\label{case_neighbored_quad}
	Q_{T,T'}^n(v_{N},w_{N}) := & h_{T}h_{T'} \bigg( GJ_{n,x}^{0,2-2s}\circ GL_{n,y} \left( \dfrac{(\hat{v}_{N,T}(1-x)-\hat{v}_{N,T'}(xy)) (\hat{w}_{N,T}(1-x)-\hat{w}_{N,T'}(xy))}{|h_{T}+yh_{T'}|^{1+2s}\,x^2} \right) \\
	& + GL_{n,x}\circ GJ_{n,y}^{0,2-2s} \left( \dfrac{(\hat{v}_{N,T}(1-xy)-\hat{v}_{N,T'}(y)) (\hat{w}_{N,T}(1-xy)-\hat{w}_{N,T'}(y))}{|xh_{T}+h_{T'}|^{1+2s}\,y^2}  \right) \bigg).
\end{align}

\subsubsection*{Separated elements ($\boldsymbol{\overline{T}\cap \overline{T'}= \emptyset}$):}
This time, the integrand is not singular. Therefore, one can directly transform the double integral to the reference square and employ tensor product Gauss-Legendre quadrature, 
which produces as the approximation of $I_{T,T'}(v_{N},w_{N}) $ the expression 
\begin{align*}
Q_{T,T'}^n(v_{N},w_{N}):= h_{T}h_{T'} ~ GL_{n,x}\circ GL_{n,y} \left( \dfrac{(\hat{v}_{N,T}(x)-\hat{v}_{N,T'}(y)) (\hat{w}_{N,T}(x)-\hat{w}_{N,T'}(y))}{|(1-x)h_{T}+ \operatorname{dist}_{T,T'} +  yh_{T'}|^{1+2s}} \right),
\end{align*}
where $ \operatorname{dist}_{T,T'}$ denotes the Euclidean distance between the elements $ T $ and $T'$. 

\subsubsection*{Complement part ($\boldsymbol{I_{T,\Omega^{c}}}$):}
The inner integral over $\Omega^{c}$ can be calculated explicitly exploiting that the functions $v_{N},w_{N} \in S_{0}^{p,1}(\mathcal{T}_{\gamma}) $ vanish outside of $\Omega = (-1,1)$. The outer integral can be transformed to the reference element $\widehat{T}$, which gives
\begin{align}\label{case_complementary}
	I_{T,\Omega^{c}}(v_{N},w_{N}) := & \frac{h_{T}}{2s} \int\limits_{\widehat{T}} \dfrac{\hat{v}_{N,T}(x)\hat{w}_{N,T}(x)}{| \operatorname{dist}_{T,\revision{\{-1\}}} +  xh_{T}|^{2s}} +  \dfrac{\hat{v}_{N,T}(x)\hat{w}_{N,T}(x)}{| \operatorname{dist}_{T,\revision{\{1\}}} +  (1-x)h_{T}|^{2s}} ~dx.
\end{align}
If $T$ is an interior element,  i.e., $\overline{T}\cap\partial\Omega = \emptyset$,  we employ Gauss-Legendre quadrature 
\begin{align*}
	Q^n_{T,\Omega^{c}}(v_{N},w_{N}) := & \frac{h_{T}}{2s} \left( GL_{n}\left( \dfrac{\hat{v}_{N,T}(x)\hat{w}_{N,T}(x)}{| \operatorname{dist}_{T,\revision{\{-1\}}} +  xh_{T}|^{2s}} \right) + GL_{n}\left( \dfrac{\hat{v}_{N,T}(x)\hat{w}_{N,T}(x)}{| \operatorname{dist}_{T,\revision{\{1\}}} +  (1-x)h_{T}|^{2s}} \right) \right).
\end{align*}
For $\overline{T} \cap \partial \Omega = \{-1\}$, we set 
\begin{align}\label{quad_comp_schnitt1}
	Q^n_{T,\Omega^{c}}(v_{N},w_{N}) := & \frac{h_{T}}{2s}  \left( GJ_{n}^{0,2-2s}\left( \dfrac{\hat{v}_{N,T}(x)\hat{w}_{N,T}(x)}{x^{2} h_{T}^{2s}} \right) + GL_{n}\left( \dfrac{\hat{v}_{N,T}(x)\hat{w}_{N,T}(x)}{| \operatorname{dist}_{T,\revision{\{1\}}} +  (1-x)h_{T}|^{2s}} \right) \right),
\end{align}
and for $\overline{T} \cap \partial \Omega = \{1\}$
\begin{align}\label{quad_comp_schnitt2}
	Q^n_{T,\Omega^{c}}(v_{N},w_{N}) := & \frac{h_{T}}{2s} \left( GL_{n}\left( \dfrac{\hat{v}_{N,T}(x)\hat{w}_{N,T}(x)}{| \operatorname{dist}_{T,\revision{\{-1\}}} +  xh_{T}|^{2s}} \right) + GJ_{n}^{2-2s,0}\left( \dfrac{\hat{v}_{N,T}(x)\hat{w}_{N,T}(x)}{(1-x)^{2}h_{T}^{2s}} \right) \right).
\end{align}
Now, having defined $Q^n_{A,B}(v_{N},w_{N})$ for all cases of integrals $I_{A,B}(v_{N},w_{N}) $, we obtain the approximated bilinear form as
\begin{align}\label{bilinear_aprox_definition}
	\widetilde{a}_{n}(v_{N},w_{N}):= \frac{C(s)}{2} \bigg( \sum\limits_{T \in \mathcal{T}_{\gamma}} \sum\limits_{T' \in \mathcal{T}_{\gamma}} Q_{T,T'}^n(v_{N},w_{N}) + 2  \sum\limits_{T \in \mathcal{T}_{\gamma}} Q^n_{T,\Omega^{c}}(v_{N},w_{N}) \bigg).
\end{align}
%
%
%
%

\subsection{\revision{Stability of the quadrature rule}}
\label{sec:stability}
Positivity of the kernel function $(x,y) \mapsto |x-  y|^{-1-2 s }$ and the Gauss-Legendre/Gauss-Jacobi
weights as well as exactness of the  Gauss-Legendre/Gauss-Jacobi quadrature allow us to prove the 
following stability result: 
\begin{lemma}
\label{lemma:stability}
Let ${\mathcal T}_\gamma$ be a $\gamma$-shape regular mesh. Then, the following holds:
\begin{enumerate}[nosep, label=(\roman*),leftmargin=*]
\item 
\label{item:lemma:stability-i}
For all $ n \ge 1$ and all $T,T' \in {\mathcal T}_\gamma \cup \{\Omega^c\}$, we have $Q^n_{T,T'}(u,u) \ge 0$.
\item
\label{item:lemma:stability-ii}
Let $n \ge p$ and $u \in S^{p,1}_0({\mathcal T}_\gamma)$. Then
$\widetilde a_n(u,u) =  0$ implies $u = 0$. In particular, the stiffness matrix $A$ in (\ref{galerkin_approx_linear_system})
is symmetric positive definite. 
\end{enumerate}
Furthermore, there is $C_{coer} > 0$ depending only on $\gamma$ and $s$ such that for 
all $u \in S^{p,1}_0({\mathcal T}_\gamma)$ 
the following assertions hold: 
\begin{enumerate}[nosep, label=(\roman*),leftmargin=*]
\setcounter{enumi}{2}
\item 
\label{item:lemma:stability-iii}
(Identical elements)
For $n \ge p$ and $T \in {\mathcal T}_\gamma$: $Q^n_{T,T}(u,u)  = I_{T,T}(u,u)$. 
\item 
\label{item:lemma:stability-iv}
(Adjacent elements)
For $n \ge p+1$ and $(T,T') \in {\mathcal T}_\gamma \times ({\mathcal T}_\gamma \cup \{\Omega^c\})$ 
with $T \ne T'$ and $\overline{T} \cap \overline{T'} \ne \emptyset$: $Q^n_{T,T'}(u,u)  \ge C_{coer}I_{T,T'}(u,u) \ge 0$. 
\item 
\label{item:lemma:stability-v}
(Separated elements)
For $n \ge p+1$ and $(T,T') \in {\mathcal T}_\gamma \times ({\mathcal T}_\gamma \cup \{\Omega^c\})$ 
with $\overline{T} \cap \overline{T'} = \emptyset$: $Q^n_{T,T'}(u,u)  \ge C_{coer} I_{T,T'}(u,u) \ge 0$. 
\end{enumerate}
\end{lemma}
\begin{proof}

\emph{Proof of \ref{item:lemma:stability-i}:} This follows from the positivity of the kernel and the Gauss-Legendre/Gauss-Jacobi weights. 

\emph{Proof of \ref{item:lemma:stability-ii}:} From \ref{item:lemma:stability-i}, we get for $u \in S^{p,1}_0({\mathcal T}_\gamma)$ with $ \widetilde{a}_n(u,u) = 0 $
\begin{align*}
0 & = \frac{2}{C(s)} \widetilde{a}_n(u,u)  = \sum_{T \in {\mathcal T}_\gamma \cup \{\Omega^c\} } \sum_{T' \in {\mathcal T}_\gamma \cup \{\Omega^c\}} Q^n_{T,T'}(u,u) 
\stackrel{\ref{item:lemma:stability-i}}{\ge} \sum_{T \in {\mathcal T}_\gamma} Q^n_{T,T} (u,u) 
\stackrel{\ref{item:lemma:stability-iii}}{=} \sum_{T \in {\mathcal T}_\gamma} I_{T,T} (u,u)  \ge 0. 
\end{align*}
Hence, $|u|_{H^s(T)} = 0$ for each $T \in {\mathcal T}_\gamma$ so that $u$ is constant on each element. By continuity of $u$, 
it is constant on $\Omega$,  and the boundary conditions then imply $u = 0$.

\emph{Proof of \ref{item:lemma:stability-iii}:} For  $n \ge p$, the univariate Gauss-Jacobi quadrature in (\ref{case_idpanel}) 
is exact for polynomials of degree $2p-1$. Inspection of (\ref{case_idpanel}) shows that the argument is the square of a polynomial
of degree $p-1$ in each variable. 

\emph{Proof of \ref{item:lemma:stability-iv}:} For $n \ge p+1$, the univariate Gauss-Jacobi quadratures in 
(\ref{case_neighbored_quad}) are exact for polynomials of degree $2p+1$. We study the cases 
$(T,T') \in {\mathcal T}_\gamma \times {\mathcal T}_\gamma$ and 
$(T,T') \in {\mathcal T}_\gamma \times \{\Omega^c\}$ separately, starting with 
$(T,T') \in {\mathcal T}_\gamma \times {\mathcal T}_\gamma$.  
We only consider the first term  in
(\ref{case_neighbored_quad}), the other one being handled analogously. Let $u \in S^{p,1}_0({\mathcal T}_\gamma)$. 
For the pull-backs $\hat u_T$, $\hat u_{T'}$ to the reference element $\widehat{T}$ of the functions $u|_T$, $u|_{T'}$, we get by continuity of $u$ 
at $\overline{T} \cap \overline{T'}$ that $\hat u_T(1) = \hat u_{T'}(0)$. Hence, 
\begin{equation*}
U(x,y):= \frac{\hat u_T(1-x) - \hat u_{T'}(xy)}{x} 
\end{equation*}
is a polynomial of degree $p-1$ in $x$ and of degree $p$ in $y$. 
Using the positivity of the quadrature weights and the exactness of the quadrature rules ($U^2$ is a polynomial of degree 
$2p \leq 2p+1$ in each variable) 
\begin{align*}
&\hspace{-10mm} h_T h_{T'} GJ^{0,2-2s}_{n,x}\circ GL_{n,y} \left(\frac{(\hat u_T(1-x) - \hat u_{T'}(xy))^2}{x^2 (h_T + y h_{T'})^{1+2s}}\right)  \\
 &\qquad=h_T h_{T'}
GJ^{0,2-2s}_{n,x}\circ GL_{n,y} \left( U^2(x,y) (h_T + y h_{T'})^{-(1+2s)}\right)  \\
&\qquad \ge 
h_T h_{T'} GJ^{0,2-2s}_{n,x}\circ GL_{n,y} \left( U^2(x,y) (h_T + h_{T'})^{-(1+2s)}\right) \\
&\qquad= 
h_T h_{T'} \int_{x \in \widehat{T}} 
\int_{y \in \widehat{T}} U^2(x,y) (h_T+h_{T'})^{-(1+2s)} x^{2-2s}\, dy\,dx  \\
&\qquad\ge (1 + h_{T'}/h_T)^{-(1+2s)} h_T h_{T'}
\int_{x \in \widehat{T}} 
\int_{y \in \widehat{T}} U^2(x,y) (h_T+y h_{T'})^{-(1+2s)} x^{2-2s}\, dy\,dx,  
\end{align*}
where we used in the last inequality that $ h_{T}^{-(1+2s)} \ge (h_T+y h_{T'})^{-(1+2s)}$.
We conclude  in view of (\ref{case_neighbored})
\begin{align*}
Q^{n}_{T,T'}(u,u) \ge C_{coer} \, I_{T,T'}(u,u), 
\end{align*}
where $C_{coer}:= \inf \{ (1 + h_{T'}/h_T)^{-(1+2s)}\,|\, T \in {\mathcal T}_\gamma, T' \mbox{ adjacent to $T$}\}$ depends only the shape regularity constant $\gamma$ and $s$.
The case $(T,T') \in {\mathcal T}_\gamma \times \{\Omega^c\}$ leads to two terms of the form \eqref{quad_comp_schnitt1} or \eqref{quad_comp_schnitt2}. One term can always be analyzed in similar fashion as above and the other one can be treated as in the following case \ref{item:lemma:stability-v}.

\emph{Proof of  \ref{item:lemma:stability-v}:} This is handled similarly to the case of adjacent elements 
in \ref{item:lemma:stability-iv}. We consider only the case $(T,T') \in {\mathcal T}_\gamma \times {\mathcal T}_\gamma$, 
the case 
$(T,T') \in {\mathcal T}_\gamma \times \{\Omega^c\}$ is handled similarly. 

With $\hat u_T$, $\hat u_{T'}$ as above and using that polynomials of degree $2p+1$ are integrated 
exactly for $n \ge p+1$ we estimate 
\begin{align*}
Q^n_{T,T'}(u,u)  &= 
h_T h_{T'} GL_{n,x} \circ GL_{n,y} \left((\hat u_T(x) - \hat u_{T'}(y))^2 ((1-x) h_T + \operatorname{dist}_{T,T'} + y h_{T'})^{-(1+2s)}\right) \\
& \ge h_T h_{T'} GL_{n,x} \circ GL_{n,y} \left((\hat u_T(x) - \hat u_{T'}(y))^2 (h_T + \operatorname{dist}_{T,T'} + h_{T'})^{-(1+2s)}\right) \\
& =  h_T h_{T'}
\int_{x \in \widehat{T}} \int_{y \in \widehat{T}}  (\hat u_T(x) - \hat u_{T'}(y))^2 (h_T+\operatorname{dist}_{T,T'} +h_{T'})^{-(1+2s)}\, dy\, dx \\
& \ge  
\left(\frac{\operatorname{dist}_{T,T'}}{h_T + \operatorname{dist}_{T,T'} + h_{T'}}\right)^{1+2s}  h_T h_{T'}
\int_{x \in \widehat{T}} \int_{y \in \widehat{T}}  \dfrac{(\hat u_T(x) - \hat u_{T'}(y))^2}{((1-x) h_T + \operatorname{dist}_{T,T'} \revision{+} y h_{T'})^{1+2s}} \, dy\, dx \\ 
& \ge C_{coer} I_{T,T'}(u,u), 
\end{align*}
for a $C_{coer}>0$ that depends solely on the shape regularity constant $\gamma$ and $s$. 
\end{proof}
\begin{remark}
\label{remk:sharper-quadrature-rule}
The proof shows that the condition $n \ge p+1$ for the case of adjacent elements could be weakened in that $p$ points suffice in one variable whereas 
$p+1$ point should be used in the other one. 
\eremk
\end{remark}
\begin{corollary}
\label{cor:stability}
Let ${\mathcal T}_\gamma$ be a $\gamma$-shape regular mesh. There is $c_{coer} > 0$ depending only on $\gamma$ and $s$
such that for $n \ge p+1$ 
\begin{align}
\label{eq:cor:stability-10}
\widetilde{a}_n(u,u) & \ge c_{coer}\|u\|^2_{\widetilde{H}^s(\Omega)} \quad \forall u \in S^{p,1}_0(\mathcal{T}_\gamma). 
\end{align}
\end{corollary}
\begin{proof} 
We write 
\begin{align*}
\widetilde{a}_n(u,u)  = \frac{C(s)}{2} \sum_{T \in {\mathcal T}_\gamma \cup \{\Omega^c\} } \sum_{T' \in {\mathcal T}_\gamma \cup \{\Omega^c\}} Q^n_{T,T'}(u,u) 
\end{align*}
and use Lemma~\ref{lemma:stability} to bound $Q^n_{T,T'}(u,u) \ge C_{coer} I^n_{T,T'}(u,u)$ for a $C _{coer} > 0$ depending only on $\gamma$ and $s$. 
\end{proof}
\begin{remark}
\label{remk:number-quadrature-points}

\begin{enumerate*}[label=(\roman*)]
\item 
Lemma~\ref{lemma:stability} shows that it suffices to use $n \ge p$ quadrature points for the quadrature $Q^n_{T,T}$ to ensure solvability
of the linear system (\ref{galerkin_approx_linear_system}). The condition $n \ge p+1$ stipulated in Corollary~\ref{cor:stability} leads to uniform 
(in $p$ and $\mathcal{T}_\gamma$)
coercivity. 
\item 
Remark~\ref{remk:sharper-quadrature-rule} shows that for adjacent elements a ``mixed'' quadrature order could be employed to slightly
reduce the number of quadrature points. 
\item 
The stability result Corollary~\ref{cor:stability} exploits positivity of the kernel and weights as well as a certain exactness property
of the Gauss-Legendre/Gauss-Jacobi quadratures. One can avoid exploiting these properties and rely on 
a perturbation argument that uses consistency error estimates for the quadratures and the coercivity of the continuous bilinear form
$a(\cdot,\cdot)$. This approach, which results in 
the stronger requirement $n \ge p + O(\log ((p+1) (\#{\mathcal T}_\gamma +1)))$ is discussed in Lemma~\ref{lemma:uniform_coercivity} below. 
\eremk
\end{enumerate*}
\end{remark}
\section{Proof of Theorem~\ref{main_result}} \label{ch:proof}
The proof is based on the classical Strang Lemma, see, e.g., \cite[Chap.~3]{braess2007finite}. In the present setting, it takes the following form: 

\begin{lemma}[First Strang Lemma]
\label{lemma:strang}
Let $\mathcal{T}$ be a mesh on $\Omega$ and let $\widetilde{\alpha}_n > 0$ be such that $\widetilde{a}_n$ satisfies 
	\begin{align}
		\widetilde{\alpha}_n \|v_{N}\|^{2}_{\widetilde{H}^{s}(\Omega)} \le \widetilde{a}_{n}(v_{N},v_{N}) ~~~~\text{for all } v_{N} \in S^{p,1}_0({\mathcal{T}}).
	\end{align} 
Then, with the continuity constant $C_a$ of the bilinear form $a$, the difference $u - \widetilde{u}_{N,n}$ between the solutions
	$u \in \widetilde{H}^{s}(\Omega)$ of \eqref{weak_formulation} and $\widetilde{u}_{N,n} \in S^{p,1}_0(\mathcal{T})$ of \eqref{eq:fully_discrete} 
satisfies 
	\begin{align*}
		\|u-\widetilde{u}_{N,n}\|_{\widetilde{H}^{s}(\Omega)} \le&~ \left(1 + \frac{C_a}{\widetilde{\alpha}_n}\right) \bigg( \inf\limits_{v\in S^{p,1}_0(\mathcal{T})} \bigg( \|u-v\|_{\widetilde{H}^{s}(\Omega)} + \sup\limits_{w \in S^{p,1}_0(\mathcal{T})} \dfrac{|a(v,w)-\widetilde{a}_{n}(v,w)|}{\|w\|_{\widetilde{H}^{s}(\Omega)}} \bigg) \notag \\
		&\quad~+ \sup\limits_{w \in S^{p,1}_0(\mathcal{T})}\dfrac{|l(w)-\widetilde{l}_{n}(w)|}{\|w\|_{\widetilde{H}^{s}(\Omega)}}  \bigg). 
	\end{align*}
\end{lemma}

Lemma~\ref{lemma:strang} indicates that we have to show lower bounds for the coercivity of $\widetilde{a}_n(\cdot,\cdot)$
as well as derive bounds for the consistency errors  $|l(w)-\widetilde{l}_{n}(w)| $ and $|a(v,w)-\widetilde{a}_{n}(v,w)|$. This is the subject of the following two lemmas, whose proofs are postponed to Section~\ref{ch:consistency_errors}.  
\begin{lemma}[Consistency error for $l$]
\label{consistency_error_lemma_functional}
Let $f$ be analytic in $\overline{\Omega}$, and let $\mathcal{T}_{\gamma}$ be a $\gamma$-shape regular mesh. 
Let $ l(v) = \skp{f,v}_{L^{2}(\Omega)}$ and let its approximation
 $ \widetilde{l}_{n}(\cdot)$ be defined by \eqref{functional_aprox_definition}.
Then, there exists a constant $\rho > 1$ depending only on $f$ such that
	\begin{align}
	|l(v)-\widetilde{l}_{n}(v)| \le C_{s,f} \rho^{p-2n+1}p\|v\|_{\widetilde{H}^{s}(\Omega)} & \qquad \text{for all } ~ v \in S_{0}^{p,1}(\mathcal{T}_{\gamma}),
	\end{align}
where $C_{s,f} > 0$ is a constant  that depends only on $s$ and $f$.
\end{lemma}

\begin{lemma}[Consistency error for $a$]
\label{consistency_error_lemma_bilinear}
 Let  $\mathcal{T}_{\gamma}$ be a $\gamma$-shape regular mesh, $ a(\cdot,\cdot) $ be the bilinear form of \eqref{weak_formulation} and $ \widetilde{a}_{n}(\cdot,\cdot) $ be its approximation \eqref{bilinear_aprox_definition}. Then, there exists a constant $\rho > 1$ that depends only on the shape regularity constant $\gamma$  such that for all $u \in S_{0}^{r,1}(\mathcal{T}_{\gamma}) $ and $v \in S_{0}^{p,1}(\mathcal{T}_{\gamma})$ there holds 
	\begin{align}
		|a(u,v)-\widetilde{a}_{n}(u,v)| \le C_{s,\gamma} (\# \mathcal{T}_{\gamma})^{2} \rho^{r+p-2n+1} r^{3} p^{3} \|u\|_{\widetilde{H}^{s}(\Omega)} \|v\|_{\widetilde{H}^{s}(\Omega)},
	\end{align}
	where $ C_{s,\gamma}$ is a constant that depends only on $\gamma$ and $s$.
\end{lemma}

\noindent As pointed out in Remark~\ref{remk:number-quadrature-points}, the consistency error $a - \widetilde{a}_n$ allows one to infer 
uniform coercivity by a perturbation argument: 
\begin{lemma}[Uniform coercivity]
\label{lemma:uniform_coercivity}
	Let the assumptions of Lemma~\ref{consistency_error_lemma_bilinear} hold. Then, there are constants  $\widetilde{\alpha}, \lambda_{1}, \lambda_{2}>0$  
depending only on the shape regularity constant $\gamma$ and $s$
such that for $n\ge p + \lambda_{1}\ln (p+1) + \lambda_{2}\ln (\# \mathcal{T}_{\gamma} +1)$ there holds 
	\begin{align}
		\widetilde{\alpha} \|v_{N}\|^{2}_{\widetilde{H}^{s}(\Omega)} \le \widetilde{a}_{n}(v_{N},v_{N}) ~~~~\text{for all } v_{N} \in S^{p,1}_0(\mathcal{T}_{\gamma}).
	\end{align}
\end{lemma}

\begin{proof}
	The coercivity of $a(\cdot,\cdot)$, the triangle inequality and Lemma~\ref{consistency_error_lemma_bilinear} applied with $r=p$ give
	\begin{align*}
		\alpha\|v_{N}\|^{2}_{\widetilde{H}^{s}(\Omega)} &\le a(v_{N},v_{N}) \le \widetilde{a}_{n}(v_{N},v_{N})+ |a(v_{N},v_{N})-\widetilde{a}_{n}(v_{N},v_{N})| \\
		& \le \widetilde{a}_{n}(v_{N},v_{N}) + C_{s,\gamma} \: (\# \mathcal{T}_{\gamma})^{2} \rho^{2p-2n+1}p^{6} \|v_{N}\|^{2}_{\widetilde{H}^{s}(\Omega)}.
	\end{align*}
	As the second term on the right-hand side tends to zero for $n \rightarrow \infty$, we may ensure for 
	$n\ge p + \lambda_{1} \ln (p+1) + \lambda_{2}\ln (\# \mathcal{T}_{\gamma} +1) $  with large enough constants $\lambda_1,\lambda_2$ that 
	\begin{align}\label{eq:coerc_tmp1}
		C_{s,\gamma} \: (\#\mathcal{T}_{\gamma})^{2} p^{6}\rho^{1-2\lambda_{1} \ln (p+1)-2\lambda_{2} \ln (\# \mathcal{T}_{\gamma}+1)} \le \dfrac{\alpha}{2}
	\end{align}
	so that coercivity of $\widetilde{a}_n$ follows with coercivity constant $\widetilde\alpha:= \alpha/2$. To give more details: we note that $\lambda_1$, $\lambda_2$ can be chosen independently of  $p$ and $\# \mathcal{T}_{\gamma}$ as
	\begin{itemize}
		\item $ \lambda_{1} \ge \dfrac{3}{\ln(\rho)} ~~~~~~~~~~~~~~~~~~~~~~~~~~~~~~~~~~~~~~~~~~~\hspace{1.6mm} \Longrightarrow~~~ (p+1)^{6-2\lambda_{1} \ln(\rho)} \le 1$,
		\item $ \lambda_{2} \ge \dfrac{2}{\ln(\rho)} ~~~~~~~~~~~~~~~~~~~~~~~~~~~~~~~~~~~~~~~~~~~\hspace{1.6mm} \Longrightarrow~~~ (\# \mathcal{T}_{\gamma} +1)^{2-\lambda_{2} \ln(\rho)} \le 1$,
		\item $ \lambda_{2} \ge \max\left( \dfrac{\ln(2\rho ~C_{s,\gamma}) -\ln(\alpha)}{\ln(\rho)\ln(2)} ,0\right) ~~~~~~\hspace{8.4mm}\Longrightarrow~~~ C_{s,\gamma} \: \rho \: (\# \mathcal{T}_{\gamma} +1)^{-\lambda_{2}\ln(\rho)} \le \dfrac{\alpha}{2}$,
	\end{itemize}
	which directly gives \eqref{eq:coerc_tmp1}.
\end{proof}

\begin{proof}[Proof of Theorem~\ref{main_result}]
\emph{Proof of (\ref{ungl_main_result}):} 
Under the assumptions made, we can apply the stability result Corollary~\ref{cor:stability} with $\mathcal{T}_{\gamma}=\mathcal{T}_{geo,\sigma}^{L}$ noting that $\# \mathcal{T}_{geo,\sigma}^{L} = 2L+2$. Hence, for $r \in \{1,\dots,p\}$, we can use the First Strang Lemma to estimate
\begin{align}\label{has_to_be_bounded}
	\|u-\widetilde{u}_{N,n}\|_{\widetilde{H}^{s}(\Omega)} \le&~ C \bigg( \inf\limits_{u_{r}\in S^{r,1}_0(\mathcal{T}^{L}_{geo,\sigma})} \bigg( \|u-u_{r}\|_{\widetilde{H}^{s}(\Omega)} + \sup\limits_{w \in S^{p,1}_0(\mathcal{T}^{L}_{geo,\sigma})} \dfrac{|a(u_{r},w)-\widetilde{a}_{n}(u_{r},w)|}{\|w\|_{\widetilde{H}^{s}(\Omega)}}\bigg)  \notag \\
	&~\quad+ \sup\limits_{\omega \in S^{p,1}_0(\mathcal{T}^{L}_{geo,\sigma})}\dfrac{|l(w)-\widetilde{l}_{n}(w)|}{\|w\|_{\widetilde{H}^{s}(\Omega)}}\bigg). 
\end{align}
\noindent Taking $u_{r} \in S^{r,1}_0({\mathcal T}^L_{geo,\sigma})$ as the $hp$-FEM approximation of (\ref{galerkin_approx}) 
for the space $S_{0}^{r,1}(\mathcal{T}^L_{geo,\sigma})$, we get from Proposition~\ref{theorem_galerkin} for the first term
\begin{align}
	\|u-u_{r}\|_{\widetilde{H}^{s}(\Omega)} \le C e^{-br} + C_{\varepsilon} \sigma^{(1/2-\varepsilon)L} .
\end{align}
Lemma \ref{consistency_error_lemma_bilinear} and the \emph{a priori} estimate $\|u_{r}\|_{\widetilde{H}^{s}(\Omega)} \le C\|f\|_{L^{2}(\Omega)}$ lead to
	\begin{align}
	 \sup\limits_{w \in S^{p,1}_0(\mathcal{T}^{L}_{geo,\sigma})} \dfrac{|a(u_{r},w)-\widetilde{a}_{n}(u_{r},w)|}{\|w\|_{\widetilde{H}^{s}(\Omega)}}\le
	C_{s,\sigma,f} \: L^{2} \rho^{r+p-2n+1}r^{3} p^{3}.
\end{align}
Finally, Lemma \ref{consistency_error_lemma_functional} provides
\begin{align}
	\sup\limits_{w \in S^{p,1}_0(\mathcal{T}^{L}_{geo,\sigma})}\dfrac{|l(w)-\widetilde{l}_{n}(w)|}{\|w\|_{\widetilde{H}^{s}(\Omega)}} \le C_f \: \rho^{p-2n+1}p.
\end{align}
This proves the convergence result (\ref{ungl_main_result}).

\emph{Proof of (\ref{ungl_main_result-DOF}):} Follows from (\ref{ungl_main_result}) by taking $r = p/2$.

\emph{Proof of the complexity estimate:}
We are left to show that, for $L \sim p \sim n$ and the basis $\mathcal{B}^{geo}= \mathcal{B}^{lin} \cup \mathcal{B}^{Leg} $ from Definition~\ref{geometric_mesh_definition}, \revision{ the number of algebraic operations to set up the linear system $Ax=b$ is $\mathcal{O}(L^{5})$,
where $A \in \mathbb{R}^{N\times N}$ with $A_{ij}=a(\varphi_{j},\varphi_{i})$ and $b \in \mathbb{R}^N$ with $b_{i}:=\skp{f,\varphi_{i}}_{L^{2}(\Omega)}$.}
\revision{The key to the proof is that the evaluation of the $p+1$ shape functions at the $n$ quadrature points always happens on the reference element $\widehat{T}$ and therefore can be precomputed. This precomputation can be realized in $\mathcal{O}( n p)$ operations
using three-term recurrence relations by noting that the integrated Legendre polynomials are orthogonal polynomials (see, e.g., \cite[(A.3), (A.9)]{karniadakis-sherwin05}). }

\revision{
	\revision{We recall that the support of the basis \revision{functions consists of} two mesh elements for $\mathcal{B}^{lin}$ and one for $\mathcal{B}^{Leg}$. Therefore, in the definition of the approximated bilinear form
\begin{align*}
	\widetilde{a}_{n}(\varphi_{i},\varphi_{j}) = \frac{C(s)}{2} \bigg( \sum\limits_{T \in \mathcal{T}_{geo,\sigma}^{L}} \sum\limits_{T' \in \mathcal{T}_{geo,\sigma}^{L}} Q_{T,T'}^n(\varphi_{i},\varphi_{j}) + 2  \sum\limits_{T \in \mathcal{T}_{geo,\sigma}^{L}} Q^n_{T,\Omega^{c}}(\varphi_{i},\varphi_{j}) \bigg),
\end{align*}
most of the summands are zero and only $\mathcal{O}(1)$ are left to calculate.}
\revision{Before we derive the stated complexity bound, we show that a direct implementation is not enough to achieve $\mathcal{O}(L^{5})$.}

\emph{\revision{Direct implementation:}} 
In terms of computational effort, \revision{the evaluation of the stiffness matrix $A_{ij}$ dominates the computation of the load vector $b_{i}$} (which is of order $\mathcal{O}(L p n)$ by the same reasoning as below). For the stiffness matrix, a naive implementation contains nested loops (starting from the outer loops) of
\begin{itemize}
\item \revision{ 2 loops over the $\mathcal{O}(p \: L)$ basis functions $ \mathcal{B}^{geo} $ (thus complexity $\mathcal{O}(p^2\: L^{2})$) with evaluation of $\mathcal{O}(1)$ quadrature formulas of the type $Q^n_{T,T'}(\varphi_{i},\varphi_{j})$ and $Q^n_{T,\Omega^{c}}(\varphi_{i},\varphi_{j})$;}
\item \revision{evaluation of each $Q^n_{T,T'}(\varphi_{i},\varphi_{j})$ and $Q^n_{T,\Omega^{c}}(\varphi_{i},\varphi_{j})$}: 2 loops over the quadrature points with complexity $\mathcal{O}(n^2)$.
\end{itemize}
In total this leads to a complexity of $\mathcal{O}(p^2 L^2) \mathcal{O}(n^2) = \mathcal{O}(L^6)$, since $L \sim p \sim n$. 
We now show that the complexity of setting up the stiffness matrix and therefore the overall complexity, can actually be reduced from $\mathcal{O}(L^6)$ to $\mathcal{O}(L^5)$.

\emph{\revision{Step~1 (blockwise assembly):}} 
We assemble the stiffness matrix blockwise. Therefore, for subsets of basis functions $\widehat{\mathcal{B}},\widetilde{\mathcal{B}} \subset \mathcal{B}$, 
we introduce the notation $\widetilde{a}_{n}(\widehat{\mathcal{B}},\widetilde{\mathcal{B}}) :=(\widetilde{a}_{n}(\revision{\varphi}_{i},\revision{\varphi}_{j}))_{\revision{\varphi}_{i} \in \widehat{\mathcal{B}},\revision{\varphi}_{j} \in \widetilde{\mathcal{B}}}$ for a matrix block. 

As $\mathcal{B} = \mathcal{B}^{lin} \cup \mathcal{B}^{Leg}$, we have 
$A = \begin{pmatrix} \widetilde a_n(\mathcal{B}^{lin},\mathcal{B}^{lin}) & \widetilde  a_n(\mathcal{B}^{lin},\mathcal{B}^{Leg})  \\  \widetilde a_n(\mathcal{B}^{Leg},\mathcal{B}^{lin}) & \widetilde a_n(\mathcal{B}^{Leg},\mathcal{B}^{Leg})  \end{pmatrix}$.
The block $\widetilde a_n(\mathcal{B}^{lin},\mathcal{B}^{lin})$ has $\mathcal{O}(L^2)$ entries (as $\# \mathcal{B}^{lin} = \mathcal{O}(L)$) that can each be calculated in $\mathcal{O}(n^2)$ operations. 
Similarly, the blocks  $\widetilde{a}_{n}(\mathcal{B}^{lin},\mathcal{B}^{Leg})$ and $\widetilde{a}_{n}(\mathcal{B}^{Leg},\mathcal{B}^{lin})$ have $\mathcal{O}(p\: L^{2})$ entries (\revision{as $\# \mathcal{B}^{Leg} = \mathcal{O}(p \: L)$}), 
which can be each calculated in $\mathcal{O}(n^{2})$ operations. 
Thus, the total complexity for the calculation of these three blocks is $\mathcal{O}(p\: L^{2}  n^{2})$ operations.
It thus remains to treat the block $\widetilde{a}_{n}(\mathcal{B}^{Leg},\mathcal{B}^{Leg})$. 

\emph{Step~2 (treatment of $\widetilde{a}_n(\mathcal{B}^{Leg},\mathcal{B}^{Leg})$):} 
Let $T,T' \in \mathcal{T}_{geo,\sigma}^{L}$ be a pair of elements. We distinguish three cases: 
the $\mathcal{O}(L)$ pairs of adjacent elements, the $\mathcal{O}(L)$ coinciding pairs $T' = T$,  and the $\mathcal{O}(L^2)$ well-separated pairs. 
\revision{For the first two cases of adjacent pairs or identical pairs, one has to consider  
$\mathcal{O}(p^{2})$ combinations of basis functions $\mathcal{B}^{Leg}$ so that the total complexity for this case is $\mathcal{O}(p^{2} L\: n^{2})$, 
which is the desired $\mathcal{O}(L^5)$ complexity.}

Therefore, let $T,T'$ be separated, i.e., $\overline T \cap \overline{T'} =  \emptyset$ and $\revision{\varphi}_T \in \mathcal{B}_T, \revision{\varphi}_{T'} \in \mathcal{B}_{T'}$ be fixed. For this case, the 
 bilinear form $\widetilde{a}_{n}(\revision{\varphi}_{T},\revision{\varphi}_{T'})$ simplifies to
\begin{align}\label{eq:effort_tmp}
	\dfrac{\widetilde{a}_{n}(\revision{\varphi}_{T},\revision{\varphi}_{T'})}{h_{T}h_{T'}} &=   GL_{n,x}\circ GL_{n,y} \left( \dfrac{\widehat{\revision{\varphi}}_{T}(x)\widehat{\revision{\varphi}}_{T'}(y)}{|(1-x)h_{T}+ \operatorname{dist}_{T,T'} +  yh_{T'}|^{1+2s}} \right) \notag \\
	&= (\omega_{i}\: \widehat{\revision{\varphi}}_{T}(x_{i}))^{\top}_{i=1,\dots,n} \cdot
	(k_{T,T'}(x_{i},y_{j}))_{i,j=1,\dots,n} \cdot
	(\omega_{j}\: \widehat{\revision{\varphi}}_{T'}(y_{j}))_{j=1,\dots,n},
\end{align}
where $k_{T,T'}(x,y):= (|(1-x)h_{T}+ \operatorname{dist}_{T,T'} +  yh_{T'}|^{1+2s})^{-1}$. 
The key observation is that the vectors $(\omega_{i}\: \widehat{\revision{\varphi}}_{T}(x_{i}))^{\top}_{i=1,\dots,n} $ and $(\omega_{j}\: \widehat{\revision{\varphi}}_{T'}(y_{j}))_{j=1,\dots,n}$ can be precomputed in $\mathcal{O}(pn)$ operations using recurrence relations since $\widehat{\revision{\varphi}}_{T}$ and $\widehat{\revision{\varphi}}_{T'} $ are the integrated Legendre polynomials on the reference element $ \widehat{T}$ and therefore independent of $T$ and $T'$. Thus, we can compute the products in (\ref{eq:effort_tmp}) as: For all pairs of separated elements $T$, $T'$ and all $\revision{\varphi}_T \in \mathcal{B}_T$, compute 
the vectors
\begin{itemize}
\item $M:=(\omega_{i}\: \widehat{\revision{\varphi}}_{T}(x_{i}))^{\top}_{i=1,\dots,n} \cdot
		(k_{T,T'}(x_{i},y_{j}))_{i,j=1,\dots,n}$ in $\mathcal{O}(n^2)$;

\item then, loop over all basis functions $\revision{\varphi}_{T'} \in \mathcal{B}_{T'}$ and compute the scalar product $M \cdot (\omega_{j}\: \widehat{\revision{\varphi}}_{T'}(y_{j}))_{j=1,\dots,n}$ in $\mathcal{O}(n)$.	
\end{itemize}
This leads to a total complexity of $\mathcal{O}(L^2 \: p \:(n^2+p\:n)) = \mathcal{O}(L^{5})$, which finishes the proof.
}

\end{proof}


\section{Consistency errors}\label{ch:consistency_errors}
In this chapter, we present the proofs for the consistency error estimates in Lemmas~\ref{consistency_error_lemma_functional} and \ref{consistency_error_lemma_bilinear}.

We start with a well-known basic error estimate for Gaussian quadrature. Recall that 
\begin{align*}
	I(\Phi):= \int_{\widehat{T}}\Phi(x)\omega(x) ~dx \approx \sum\limits_{i=1}^{n}  \omega_{i}  \Phi(x_{i})=: G_n(\Phi),
\end{align*}
with $\sum_i \omega_i =C_\omega := \int_{\widehat{T}} \omega dx$
and that the numerical integration is exact for $\Pi \in \mathcal{P}_{2n-1}(\widehat{T})$. Thus, for an arbitrary polynomial $\Pi \in \mathcal{P}_{2n-1}(\widehat{T})$ we get
(using also the positivity of the weights $\omega_i$)
\begin{align*}
	E_{n} &:= |I(\Phi)-G_{n}(\Phi)| = |I(\Phi-\Pi)-G_{n}(\Phi-\Pi)| \\
	&\le C_\omega \|\Phi-\Pi\|_{L^{\infty}(\widehat{T})} + \sum_{i=1}^n\omega_{i}\: \|\Phi-\Pi\|_{L^{\infty}(\widehat{T})} \le 2C_\omega \|\Phi-\Pi\|_{L^{\infty}(\widehat{T})},
\end{align*}
which gives the best approximation estimate
\begin{align}\label{1d_quad_approx}
	E_{n} \le 2C_\omega \inf_{\Pi \in \mathcal{P}_{2n-1}(\widehat{T})} \|\Phi-\Pi\|_{L^{\infty}(\widehat{T})}.
\end{align}

\noindent By tensorization, this result for univariate Gaussian quadrature can be extended to the $2d$-case. We consider the special case $\omega \equiv 1$ and 
for 
\begin{align*}
	I^{2D}(\Phi):=\int_{\widehat{T}} \int_{\widehat{T}} \Phi(x,y) \,dy\,dx \approx 
	G^{2D}_{n}(\Phi):= G_{n,\revision{x}} \circ G_{n,\revision{y}}(\Phi)  = G_{n,\revision{y}} \circ G_{n,\revision{x}}(\Phi) = \revision{\sum\limits_{i,j=1}^{n}} \omega_{i}\omega_{j}\Phi(x_{i},y_{j})
\end{align*}
\revision{we estimate the error} using $C_\omega = 1$ for $\omega \equiv 1$: 
\begin{align}\label{2d_to_1d}
	E^{2D}_{n}& := |I^{2D}(\Phi)-G^{2D}_{n}(\Phi)| \\
	&= \left| \int_{\widehat{T}} \int_{\widehat{T}} \Phi(x,y) \,dy - G_{n,\revision{y}}(\Phi(x,\cdot))\,dx \right| 
	+ \left| \int_{\widehat{T}}  G_{n,\revision{y}}(\Phi(x,\cdot)) \,dx - G_{n,\revision{y}} \circ G_{n,\revision{x}}(\Phi) \right| \notag \\
\nonumber 
	&\le \sup\limits_{x \in \widehat{T}} |I(\Phi(x,\cdot))-G_{n,\revision{y}}(\Phi(x,\cdot))| +  \left |G_{n,\revision{y}} \left( \int_{\widehat{T}} \Phi(x,\cdot)\,dx - G_{n,\revision{x}}(\Phi) \right)\right|  \nonumber 
\\
\nonumber 
	&\le \sup\limits_{x \in \widehat{T}} |I(\Phi(x,\cdot))-G_{n,\revision{y}}(\Phi(x,\cdot))| 
         +  \sum_{i=1}^n \revision{\omega_i}  \left| \int_{\widehat{T}} \Phi(x,\revision{y_i})\,dx - G_{n,\revision{x}}(\Phi(\cdot,\revision{y_i})) \right|  \\
\nonumber 
& \stackrel{\sum_i \omega_i = 1}{\leq} 
	\sup\limits_{x \in \widehat{T}} |I(\Phi(x,\cdot))-G_{n,\revision{y}}(\Phi(x,\cdot))| 
         +  \sup_{y \in \widehat{T}} \left| I(\Phi(\cdot,y)) - G_{n,\revision{x}}(\Phi(\cdot,y)) \right|.  
\end{align}

In view of (\ref{1d_quad_approx}), these two univariate integration errors are estimated by best approximation errors. 
For analytic integrands, the best approximation errors will be quantified in Proposition~\ref{proposition_best_approx}.

\begin{definition}[Bernstein ellipse]
	 For $\rho > 1 $, we define the Bernstein ellipse $\mathcal{E}_\rho$ and its scaled version $\widehat{\mathcal{E}}_\rho$ by 
	\begin{align}
		\mathcal{E}_{\rho}&:=\{ z \in \mathbb{C}:|z-1|+|z+1|<\rho+\rho^{-1} \}, \\ 
\label{eq:transformed-ellipse}
		\widehat{\mathcal{E}}_{\rho}&:= \revision{{F}}^{-1}_{(-1,1)}(	\mathcal{E}_{\rho}),
	\end{align}
\revision{where $F_{(-1,1)}:\mathbb{C} \rightarrow \mathbb{C}$, $x \mapsto 2x-1$ is the affine map transforming $(-1,1)$ to $(0,1)$. }
	\revision{We note that the focal points of $ \widehat{\mathcal{E}}_{\rho} $ are $0$ and $1$.}
	
\end{definition}

\begin{proposition}\label{proposition_best_approx}
	Let $\Phi$ be holomorphic on $\widehat{\mathcal{E}}_{\widetilde{\rho}}$, $\widetilde\rho > 1$. Then, for every $1<\rho<\widetilde{\rho}$, we have
	\begin{align}
		\inf\limits_{v \in \mathcal{P}_{n}} \|\Phi-v\|_{L^{\infty}(0,1)} \le \dfrac{2}{\rho-1}\rho^{-n} \|\Phi\|_{L^{\infty}(\widehat{\mathcal{E}}_{\rho})}.
	\end{align}
\end{proposition}
\begin{proof}
	This proposition is just a transformed version of \cite[Chap.~{7}, Thm.~{8.1}]{devore1993constructive}.
	
\end{proof}

With this estimate for the best approximation error, we obtain exponential convergence for the quadrature errors.

\begin{lemma}\label{lemma_quadabsch}		
Let $\widetilde{\rho} > 1$. 
\begin{enumerate}[nosep, label=(\roman*),leftmargin=*]
\item 
 Let $\Phi:\widehat{\mathcal{E}}_{\widetilde{\rho}}  \rightarrow \mathbb{C}$ be holomorphic. Then, for every $1<\rho<\widetilde{\rho}$, the quadrature error can be estimated by
		\begin{align}\label{quad_err_1D}
			|I(\Phi)-G_{n}(\Phi)| \le C\rho^{-2n+1} \|\Phi\|_{L^{\infty}(\widehat{\mathcal{E}}_{\rho})},
		\end{align}
		where the constant $C$ is independent of $n$ \revision{ and $\Phi$.}
\item 	
Let $\Phi: \widehat{\mathcal{E}}_{\widetilde{\rho}} \times \widehat{\mathcal{E}}_{\widetilde{\rho}} \rightarrow \mathbb{C}$ be such that 
for each  $y \in (0,1)$ the function $\Phi(\cdot,y)$ is holomorphic on $\widehat{\mathcal{E}}_{\widetilde{\rho}}$ and such that for each $x \in (0,1)$, $\Phi(x,\cdot)$ is holomorphic on $\widehat{\mathcal{E}}_{\widetilde{\rho}}$.
Then, for every $1<\rho<\widetilde{\rho}$, the quadrature error can be estimated by
		\begin{align}\label{quad_err_2D}
			|I^{2D}(\Phi)-G^{2D}_{n}(\Phi)| \le C\rho^{-2n+1} \left( \sup\limits_{y\in (0,1)}\|\Phi(\cdot,y)\|_{L^{\infty}(\widehat{\mathcal{E}}_{\rho})} + \sup\limits_{x\in (0,1)}\|\Phi(x,\cdot)\|_{L^{\infty}(\widehat{\mathcal{E}}_{\rho})} \right),
		\end{align}
		where the constant $C$ is independent of $n$ \revision{ and $\Phi$.}
\end{enumerate}
\end{lemma}

The norms in the previous estimates do not involve the $ \|\cdot\|_{\widetilde{H}^{s}(\Omega)}$ norm required in the Strang Lemma. This is achieved with 
an inverse estimate or a Poincar\'e type estimate.

\begin{lemma}\label{lemma_norm_approx}
 Let $\mathcal{I}  = (x_\ell, x_\ell + h_{\mathcal I})\subset \mathbb{R}$ be an interval with diameter $h_{\mathcal{I}}:= \operatorname{diam}(\mathcal{I}) < \infty$. 
\begin{enumerate}[nosep, label=(\roman*),leftmargin=*]
\item 
There is a constant independent of $\mathcal{I}$ such that for every $\rho>1$ and $p \in \mathbb{N}$ there holds for all polynomials $v \in \mathcal{P}_{p}(\mathcal{I}) $ \revision{ and their pullbacks $\hat v := v \circ F_{\mathcal{I}}$ }
	\begin{align}
		\left\|\dfrac{d}{dx}\hat{v}\right\|_{L^{\infty}(\widehat{\mathcal{E}}_{\rho})} &\le C\rho^{p}p^{3}h_{\mathcal{I}}^{s-1/2}|v|_{H^{s}(\mathcal{I})} ,\label{norm_ung1}\\
		\|\hat{v}\|_{L^{\infty}(\widehat{\mathcal{E}}_{\rho})} &\le C\rho^{p} p~ h_{\mathcal{I}}^{-1/2} \|v\|_{{H}^{s}(\mathcal{I})}. \label{norm_ung5}	
	\end{align}
\item Denote ${\mathcal I}_{sym} = (x_\ell - h_{\mathcal I}, x_\ell + h_{\mathcal I})$ and let $v \in H^s({\mathcal I}_{sym})$ with $v|_{(x_\ell-h_{\mathcal I},x_\ell)} = 0$. 
Then, there is $C > 0$ depending only on $s$ such that 
	\begin{align}
		\left\|v\right\|_{L^{2}(\mathcal{I})} &\le C \: h_{\mathcal{I}}^{s} |v|_{H^{s}(\mathcal{I}_{sym})}. \label{norm_ung4}
	\end{align}
The same estimate holds for 
${\mathcal I}_{sym} = \revision{(x_\ell, x_\ell + 2 h_{\mathcal I})}$ and $v \in H^s({\mathcal I}_{sym})$ with $v|_{(x_\ell+h_{\mathcal I},x_\ell + 2 h_{\mathcal I})} = 0$. 
\end{enumerate}
\end{lemma}

\begin{proof}
 With the Bernstein inequality \cite[Chap.~4, Thm.~2.2]{devore1993constructive}
	\begin{align*}
		\|q\|_{L^{\infty}(\widehat{\mathcal{E}}_{\rho})} \le \rho^{p} \|q \|_{L^{\infty}(0,1)}~~~ \text{for all } q \in \mathcal{P}_{p}(0,1)
	\end{align*} 
	and inserting the mean $\overline{\hat{v}}:= \int_{0}^{1} \hat{v}(x)~dx $, we obtain
	\begin{align*}
			\left\|\dfrac{d}{dx}\hat{v}\right\|_{L^{\infty}(\widehat{\mathcal{E}}_{\rho})} \le C  \rho^{p} \left\|\dfrac{d}{dx}\hat{v} \right\|_{L^{\infty}(0,1)} = C  \rho^{p} \left\|\dfrac{d}{dx}(\hat{v} - \overline{\hat{v}}) \right\|_{L^{\infty}(0,1)}.
	\end{align*}
Employing inverse inequalities of Markov type, see  \cite[Thm~3.91, Thm.~3.92]{schwab1998p} together with a fractional Poincar\'e inequality, see \cite{heuer2014equivalence}, and a scaling argument, we arrive at
	\begin{align}\label{hilf123}
\left\|\dfrac{d}{dx}(\hat{v} - \overline{\hat{v}}) \right\|_{L^{\infty}(0,1)} &\leq \revision{C p^{2} \left\|\hat{v} - \overline{\hat{v}} \right\|_{L^{\infty}(0,1)}}
\leq Cp^3 \left\|\hat{v} - \overline{\hat{v}} \right\|_{L^{2}(0,1)} \\
&\leq Cp^3 \revision{|\hat{v}|}_{H^{s}(0,1)} \le C p^3  h_{\mathcal{I}}^{s-1/2} \revision{| v |}_{H^{s}(\mathcal{I})}.
	\end{align}
This shows (\ref{norm_ung1}). Inequality \eqref{norm_ung5} follows with the same arguments.

The fractional Poincar\'e inequality \eqref{norm_ung4} can be shown by a scaling argument and the compact embedding $H^s \subset L^2$; the fact that the 
seminorm appears on the right-hand side of \eqref{norm_ung4} is a consequence of the fact that $v$ is assumed to vanish on parts of ${\mathcal I}_{sym}$. 
See also \cite{acosta2017femfractional} for the proof of a closely related result. 
\end{proof}

The following lemma provides the key technical estimates for the quadrature errors appearing in the approximated bilinear and linear forms.

\begin{lemma}\label{lemma_quadrature_error}
	Let $\operatorname*{co}(T,T')$ denote the convex hull of two sets $T$ and $T'$. Let ${\mathcal T}_{\gamma}$ be a $\gamma$-shape regular mesh on $\Omega$. There exists a constant $\rho > 1 $ that depends  only on $\gamma$ and $s$ such that for all $v \in S_{0}^{r,1}(\mathcal{T}_{\gamma})$, $w \in S_{0}^{p,1}(\mathcal{T}_{\gamma})$ and $T$, $T' \in  \mathcal{T}_{\gamma} $ there holds
	\begin{align}\label{quad_error_drinnen}
		\left| I_{T,T'}(v,w) - Q^n_{T,T'}(v,w) \right| 
& \le \revision{C_{s,\rho,\gamma}} r^3 p^3\rho^{r+p-2n+1}|v|_{H^{s}(\operatorname*{co}(T,T'))}|w|_{H^{s}(\operatorname*{co}(T,T'))},  \\
	\label{quad_error_complement}
		\left| I_{T,\Omega^{c}}(v,w) - Q^n_{T,\Omega^{c}}(v,w) \right| &\le
			\revision{C_{s,\rho,\gamma}} r^3 p^3\rho^{r+p-2n+1}|v|_{\widetilde{H}^{s}(\Omega)}|w|_{\widetilde{H}^{s}(\Omega)}. 
	\end{align}
\end{lemma}
\begin{proof}
We distinguish the cases of pairs of adjacent elements, identical pairs, well-separated pairs, and combinations of elements $T$ with $\Omega^c$. 

\emph{Case of adjacent elements:}
We start with the case for adjacent elements $T \neq T'$ with $\overline{T} \cap \overline{T'} \ne \emptyset$. Due to Lemma~\ref{lemma_quadabsch} it is sufficient to estimate the $L^\infty$-norms of the integrands in \eqref{case_neighbored}. As both integrands can be treated in the same way, we only consider the first one
	\begin{align*}
		 \widehat{g}_1(x,y):= h_{T}^{-1-2s}\dfrac{(\hat{v}_{T}(1-x)-\hat{v}_{T'}(xy))}{x} \cdot \dfrac{(\hat{w}_{T}(1-x)-\hat{w}_{T'}(xy))}{x} \cdot \frac{1}{|1+yh_{T'}/h_{T}|^{1+2s}}.
	\end{align*}
Note that the first two fractions of the product on the right-hand side have removable singularities and are therefore holomorphic on $\mathbb{C}$ in each variable. 
The function $y \mapsto |1+y  h_{T'}/h_{T}|  = \sqrt{(1+y h_{T'}/h_{T})^2} > 0 $ on the closed interval $[0,1]$ and therefore has a holomorphic extension to 
an ellipse $\widehat{\mathcal{E}}_{\rho}$ for some $\rho > 1$ that solely depends on $\gamma$ since $h_{T'}/h_{T} \leq 1/\gamma$ by shape regularity.  
We conclude that $\widehat{g}_1(\cdot,y)$ is holomorphic on ${\mathbb C}$ for fixed $y \in [0,1]$ and $\widehat{g}_1(x,\cdot)$ is holomorphic on $\widehat{\mathcal E}_{\rho}$ for fixed $x \in [0,1]$. 
Using that $\hat{v}_{T}(1) = \hat{v}_{T'}(0)$, the fundamental theorem of calculus implies 
for \revision{ $(x,y) \in [0,1] \times \widehat{\mathcal E}_\rho$ and for
$(x,y) \in \widehat{\mathcal E}_\rho \times [0,1]$}
	\begin{align*}
		\bigg|\frac{1}{x}(\hat{v}_{T}(1-x)-\hat{v}_{T'}(xy))\bigg|&= \biggg|\frac{1}{x}\left(\int\limits_{0}^{x} \dfrac{d}{dz}\hat{v}_{T}(1-z)dz-\int\limits_{0}^{xy} \dfrac{d}{dz}\hat{v}_{T'}(z)dz\right)\biggg| \\
		&\le 2  \max\left( \left\|\dfrac{d}{dz}\hat{v}_{T}\right\|_{L^{\infty}(\widehat{\mathcal{E}}_{\rho})}, \left\|\dfrac{d}{dz}\hat{v}_{T'}\right\|_{L^{\infty}(\widehat{\mathcal{E}}_{\rho})}\right).
	\end{align*} 
Analogously, the same can be shown for the function $\hat{w}$. 
With Lemma~\ref{lemma_norm_approx}, this implies 
	\begin{align*}
		\sup\limits_{y\in (0,1)}\|\widehat{g}_1(\cdot,y)\|_{L^{\infty}(\widehat{\mathcal{E}}_{\rho})} &+ \sup\limits_{x\in (0,1)}\|\widehat{g}_1(x,\cdot)\|_{L^{\infty}(\widehat{\mathcal{E}}_{\rho})} \\
		&\le \revision{C_{s,\rho,\gamma}} r^3 p^3\rho^{r+p}\max(|v|_{H^{s}(T)},|v|_{H^{s}(T')})\max(|w|_{H^{s}(T)},|w|_{H^{s}(T')}).
	\end{align*}
Together with \eqref{quad_err_2D} and $ \max(|v|_{H^{s}(T)},|v|_{H^{s}(T')}) \le |v|_{H^{s}(\operatorname*{co}(T,T'))}$, this finishes the proof for the case of adjacent elements $T,T'$.
\bigskip

\emph{Case of identical elements:}
	The case $T=T'$ follows \revision{with similar} arguments. We note that in this case the integrand 
	\begin{align}
		\widehat{g}_2(x,y):= \dfrac{(\hat{u}_{T}(x)-\hat{u}_{T}(xy)) (\hat{v}_{T}(x)-\hat{v}_{T}(xy))}{|x-xy|^{2}}h_{T}^{-1-2s}
	\end{align}
is a polynomial of degree $\leq r+p-1$ and thus is integrated exactly for $n \geq \max(r,p)$.
\bigskip

\emph{Case of well-separated elements:}
	For separated elements $\overline{T} \cap \overline{T'} = \emptyset$ the integrand is continuous. Thus, by  \cite[Lem.~4.6]{melenk1998hp}, the Gaussian quadrature error can be estimated by the best approximation error for the function $\widehat{g}_3(x,y):= |\operatorname{dist}_{T,T'}+(1-x)h_{T}+y h_{T'}|^{-1-2s}$ in $L^\infty$ using polynomials of maximal degree $r_{c}:=2n-p-r-1$ and $L^2$-norms of the polynomials $\hat{v}_{T}(x)-\hat{v}_{T'}(y)$ and $\hat{w}_{T}(x)-\hat{w}_{T'}(y)$: 
	\begin{align}\label{zwischen2}
		&\bigg| I_{T,T'} (v,w) - Q^n_{T,T'}(v,w) \bigg| 	 \le C p^{2}h_{T} h_{T'} \inf_{\widehat{q} \in  \mathcal{Q}_{r_{c}}((0,1)^2)}\|\widehat{g}_3-\widehat{q}\|_{L^{\infty}(\widehat{T}\times \widehat{T})}  \cdot  \nonumber \\
&\qquad\quad \Big(\int_{\widehat{T}} \int_{\widehat{T}}(\hat{v}_{T}(x)-\hat{v}_{T'}(y))^2 dy dx\Big)^{1/2}  \Big(\int_{\widehat{T}} \int_{\widehat{T}}(\hat{w}_{T}(x)-\hat{w}_{T'}(y))^2 dy dx\Big)^{1/2}  ,
	\end{align}
 where ${\mathcal Q}_{r_c}((0,1)^2)$ denotes the tensor product space ${\mathcal P}_{r_c}(0,1) \otimes {\mathcal P}_{r_c}(0,1) = 
\operatorname{span} \{(x,y) \mapsto x^i y^j\colon 0 \leq i,j \leq r_c\}$. 
Similarly to the case of adjacent elements, the function $\widehat g$ admits a holomorphic extension to $\widehat{\mathcal{E}}_{\rho}  \times \widehat{\mathcal{E}}_{\rho}$
for some $\rho > 1$ since $\widehat{g}_3(x,y)= ((\operatorname{dist}_{T,T'}+(1-x)h_{T}+y h_{T'})^2) ^{-1/2-s}$ and the argument of $(\cdot)^{-1-2s}$ is bounded away from 
$0$ for $(x,y) \in [0,1]^2$. In fact, we only require that for each fixed $x \in [0,1]$ the function $\widehat g_3(x,\cdot)$ can be extended holomorphically  to $\widehat{\mathcal E}_{\rho}$ and 
for each fixed $ y \in [0,1]$ the function $\widehat{g}_3(\cdot,y)$  can be extended holomorphically to $\widehat{\mathcal E}_{\rho}$. 
As in the case of adjacent element, we have by shape regularity 
$h_T/\operatorname{dist}_{T,T'} \leq 1/\gamma$ and $h_{T'}/\operatorname{dist}_{T,T'} \leq 1/\gamma$.  

We may employ Proposition~\ref{proposition_best_approx} and a tensor product argument akin to that employed in (\ref{2d_to_1d}) to 
get with inequality (\ref{2d_to_1d}) the existence of $\rho > 1$ such that
	\begin{align}
	\inf_{\widehat{q} \in \mathcal{Q}_{r_{c}}((0,1)^2)}\|\widehat{g}_3-\widehat{q}\|_{L^{\infty}(\widehat{T}\times \widehat{T})} \le \revision{C_{s,\rho,\gamma}} \rho^{-r_{c}} \operatorname{dist}_{T,T'}^{-1-2s}.
	\end{align}

For the remaining terms in \eqref{zwischen2}, we transform back to the physical elements, insert the mean $\overline{v_{\operatorname*{co}(T,T')}}:= \int_{\operatorname*{co}(T,T')} v(x) dx / h_{\operatorname*{co}(T,T')}  $  over the convex hull $co(T,T')$ of $T$ and $T'$ and integrate in one variable to obtain 
\begin{align}\label{zwischen3}
&\int_{\widehat{T}} \int_{\widehat{T}}(\hat{v}_{T}(x)-\hat{v}_{T'}(y))^2 dy dx =  h_T^{-1}h_{T'}^{-1}\int_{T} \int_{T'}(v_{T}(x)-v_{T'}(y))^2 dy dx \nonumber \\
&\qquad\leq 2h_T^{-1}h_{T'}^{-1}\int_{T} \int_{T'}(v(x)-\overline{v_{\operatorname*{co}(T,T')}})^2 + (\overline{v_{\operatorname*{co}(T,T')}}-v(y))^2 dy dx \nonumber \\
&\qquad= 2h_T^{-1}\|v-\overline{v_{\operatorname*{co}(T,T')}}\|_{L^2(T)}^2 + 2h_{T'}^{-1}\|v-\overline{v_{\operatorname*{co}(T,T')}}\|_{L^2(T')}^2. 
\end{align}
Both terms can be treated in the same way, we thus only focus on the first one.
Increasing the domain of integration to 
the convex hull $\operatorname*{co}(T,T')$ and employing a Poincar\'e inequality, see \cite[Prop.~{2.2}]{heuer2014equivalence}, gives
	\begin{align}
		\|v-\overline{v_{\operatorname*{co}(T,T')}}\|^{2}_{L^{2}(T)} \leq \|v-\overline{v_{\operatorname*{co}(T,T')}}\|^{2}_{L^{2}(\operatorname*{co}(T,T'))} \le \revision{C_s} h_{\operatorname*{co}(T,T')}^{2s}|v|^{2}_{H^{s}(\operatorname*{co}(T,T'))}.
	\end{align}
 Inserting everything into \eqref{zwischen2} gives
	\begin{align}
	\left| I_{T,T'}(v,w) - Q^{n}_{T,T'}(v,w) \right| \le C p^2\rho^{r+p-2n+1}\operatorname{dist}_{T,T'}^{-1-2s}(h_T+h_{T'})h_{\operatorname*{co}(T,T')}^{2s}|v|_{H^{s}(\operatorname*{co}(T,T'))}|w|_{H^{s}(\operatorname*{co}(T,T'))}.
\end{align}
	We note that, for shape regular meshes, we can estimate 
\begin{align*}
h_T & \leq \gamma^{-1} \operatorname{dist}_{T,T'},
& 
h_{T'} & \leq \gamma^{-1} \operatorname{dist}_{T,T'},
& 
h_{\operatorname{co}(T,T')} & \leq \operatorname{dist}_{T,T'} + h_{T} + h_{T'} \leq \operatorname{dist}_{T,T'}  \left(1 + \frac{2}{\gamma}\right). 
\end{align*}
Thus, 
there holds $\operatorname{dist}_{T,T'}^{-1-2s}(h_{T'}+h_{T})h^{2s}_{\operatorname*{co}(T,T')}\le (2/\gamma) (1 +2/\gamma)^{2s}$, which concludes the argument
for the case of separated elements. 
\bigskip

\emph{Case of combination of $T$ with $\Omega^c$:}
	For the complementary part, see (\ref{case_complementary}), we consider integrals of the form
	\begin{align}\label{comp_hilf1}
		h_{T} \int\limits_{\widehat{T}} \dfrac{\hat{v}_{T}(x)\hat{w}_{T}(x)}{| \operatorname{dist}_{T,\{-1\}} +  xh_{T}|^{2s}} \,dx.
	\end{align}
	We have to distinguish two cases. If $T$ is at the left boundary, $ -1 \in \overline{T}$ and therefore $\operatorname{dist}_{T,\revision{\{-1\}}} = 0$, we can treat the singular integral (\ref{comp_hilf1}) as a one dimensional version of the adjacent case. If $\operatorname{dist}_{T,\{-1\}} > 0$, the proof uses similar techniques as the separated case.
	The only difference is that, instead of the convex hull of two elements, the convex hull of the element and the boundary point $-1$ is used and 
\cite[Prop.~{2.2}]{heuer2014equivalence} is replaced with \eqref{norm_ung4} to bound the $L^{2}$-norms 
	\begin{align}
		\|v\|_{L^{2}(T)} \le \|v\|_{L^{2}(\operatorname*{co}(T,\{-1\}))} \le C \: h^{s}_{\operatorname*{co}(T,\{-1\})} \|v\|_{H^{s}(\Omega)}. 
	\end{align}
This finishes the proof.
\end{proof}

Now, the consistency errors follow from summation of the elementwise contributions. 

\begin{proof}[Proof of  Lemma~\ref{consistency_error_lemma_bilinear}]
	With the triangle inequality, basic integration and Lemma \ref{lemma_quadrature_error} we obtain
	\begin{align}\label{proof_of_consistency_error_lemma_bilinear}
 |a(\revision{v},\revision{w})-\widetilde{a}_{n}(\revision{v},\revision{w})| &\le \sum\limits_{T \in \mathcal{T}_{\gamma}}\sum\limits_{T' \in \mathcal{T}_{\gamma}} 	\left| I_{T,T'}(v,w) - Q^{n}_{T,T'}(v,w) \right| 
		+2 \sum\limits_{T \in \mathcal{T}_{\gamma}} 	\left| I_{T,\Omega^{c}}(v,w) - Q_{T,\Omega^{c}}^n(v,w) \right| \notag  \\
		&\le \revision{C_{s,\rho,\gamma}} r^3 p^3 \rho^{r+p-2n+1}\bigg(\sum\limits_{T \in \mathcal{T}_{\gamma}} \sum\limits_{T' \in \mathcal{T}_{\gamma}} |v|_{H^{s}(\operatorname*{co}(T,T'))}|w|_{H^{s}(\operatorname*{co}(T,T'))} \notag  \\
		&\quad+2 \sum\limits_{T \in \mathcal{T}_{\gamma}} \|\revision{v}\|_{\widetilde{H}^{s}(\Omega)}\|\revision{w}\|_{\widetilde{H}^{s}(\Omega)} \bigg) \notag  \\
		&\le \revision{C_{s,\rho,\gamma}} (\#\mathcal{T}_{\gamma})^{2} r^3 p^{3} \rho^{r+p-2n+1}\|\revision{v}\|_{\widetilde{H}^{s}(\Omega)} \|\revision{w}\|_{\widetilde{H}^{s}(\Omega)},
	\end{align}
which finishes the proof.
\end{proof}

\begin{proof}[Proof of Lemma~\ref{consistency_error_lemma_functional}]
As $f$ is analytic on $[0,1]$ there exists an analytic extension to a Bernstein ellipse $\widehat{\mathcal{E}}_{\rho}$ for some $\rho>1$.
Using \eqref{quad_err_1D} of Lemma~\ref{lemma_quadabsch} gives for each element
\begin{align*}
	\left|\int_{\widehat{T}}\widehat{f}_{T}(x)\widehat{v}_{T}(x) dx-GL_{n}(\widehat{f}_{T} \widehat{v}_{T})\right| &\le C\rho^{-2n+1} \|\widehat{f}_{T} \widehat{v}_{T}\|_{L^{\infty}(\widehat{\mathcal{E}}_{\rho})} \le  \revision{C_{\rho,f}} \rho^{-2n+1} \|\widehat{v}_{T}\|_{L^{\infty}(\widehat{\mathcal{E}}_{\rho})} \\
& \stackrel{\eqref{norm_ung5}}{\leq} \revision{C_{s,\rho,f}} p \rho^{p-2n+1} h_T^{-1/2} \|v\|_{\revision{{H}^{s}(T)}}.
\end{align*} 
Summation over all elements $T \in \mathcal{T}_{\gamma}$ together with the Cauchy-Schwarz inequality gives
\begin{align*}
|l(v)-\widetilde{l}_n(v)| &\leq \sum\limits_{T \in \mathcal{T}_{\gamma}} h_T	\left|\int_{\widehat{T}}\widehat{f}_{T}(x)\widehat{v}_{T}(x) dx-GL_{n}(\widehat{f}_{T} \widehat{v}_{T})\right| \\
&\leq  \revision{C_{s,\rho,f}} p \rho^{p-2n+1}  \sum\limits_{T \in \mathcal{T}_{\gamma}}  h_T^{1/2} \|v\|_{\revision{{H}^{s}(T)}} \leq  \revision{C_{s,\rho,f}} p \rho^{p-2n+1} \revision{\sqrt{|\Omega|}}\Big( \sum\limits_{T \in \mathcal{T}_{\gamma}}  \|v\|^2_{\revision{{H}^{s}(T)}} \Big)^{1/2} \\
&\leq \revision{C_{s,\rho,f}}\; p \rho^{p-2n+1}  \|v\|_{H^{s}(\Omega)}, 
\end{align*}
which finishes the proof.
\end{proof}

\section{\revision{Outlook: the multidimensional case on shape regular meshes}} \label{sec:2D}
In this section, we discuss how the preceding $1d$-analysis can be generalized to the multidimensional case $d > 1$ for bounded polyhedral Lipschitz domains $\Omega \subset \mathbb{R}^d$. In this case, 
the weak formulation is given by: Find  $u \in \widetilde H^s(\Omega)$ such that
\begin{align}\label{weak_formulationMulti}
	a(u,v):=\dfrac{C(s,d)}{2}\int_{\mathbb{R}^d}\int_{\mathbb{R}^d}\dfrac{(u(\vec{x})-u(\vec{y}))(v(\vec{x})-v(\vec{y}))}{|\vec{x}-\vec{y}|^{d+2s}} \,d\vec{y}\,d\vec{x} = \skp{f,v}_{L^{2}(\Omega)} =: l(v)
\end{align}
for all $v \in \widetilde H^s(\Omega)$, where $C(s,d):= 2^{2s}s\Gamma(s+d/2)/(\pi^{d/2}\Gamma(1-s))$ (see, e.g., \cite{acosta2017femfractional}).
Thus, we have to numerically compute integrals of the form
\begin{align}
\label{2D_int}
	I_{S_{1},S_{2}}(v,w) &:= \int\limits_{S_{1}} \int\limits_{S_{2}} \dfrac{(v(\vec{x})-v(\vec{y}))(w(\vec{x})-w(\vec{y}))}{|\vec{x}-\vec{y}|^{d+2s}} \,d\vec{y}\, d\vec{x}, \\
\label{2D_int-a}
I_{S_{1},\Omega^c} (v,w) &:=
	 \int\limits_{S_{1}} v(\vec{x}) w(\vec{x})  \int\limits_{\Omega^c} \frac{1}{|\vec{x}-\vec{y}|^{d+2s}} \,d\vec{y}\, d\vec{x},
\end{align}
where $S_{1}$ and $S_{2}$ denote $d$-dimensional simplices. 

In the following, we will consider regular, $\gamma$-shape regular triangulations ${\mathcal T}_\gamma$ of $\Omega$, i.e., decompositions of $\Omega$ into simplices. 
$\gamma$-shape regularity means that the affine element maps $F: \widehat{S} \rightarrow S$ from the reference simplex $\widehat{S}$ to $S \in {\mathcal T}_\gamma$ with 
$\operatorname{diam} S = h_S$ 
satisfy $\|F^\prime_S \|_{L^\infty} \leq \gamma h_S$ and $\|(F^\prime)^{-1}\|_{L^\infty} \leq \gamma h_S^{-1}$. 
As usual, we set $S^{p,1}_0({\mathcal T}_\gamma):= \{u \in H^1_0(\Omega)\,|\, u|_S \in {\mathcal P}_p({\mathbb R}^d) \quad \forall S \in {\mathcal P}_p\}$, where
${\mathcal P}_p({\mathbb R}^d)$ denotes the space of $d$-variate polynomials of (total) degree $p$. We will also require the tensor-product space
${\mathcal Q}_p({\mathbb R}^d):=\operatorname{span} \{x_1^{\alpha_1}\cdots x_d^{\alpha_d}\,|\, 0 \leq \alpha_1,\ldots,\alpha_d \leq p\}$. 
\subsection{Quadrature on pairs of simplices} 
In the present case of shape regular triangulations, techniques developed in \cite{chernov-schwab12} can be adapted to numerically integrate \eqref{2D_int}. Similarly to the case $d=1$ in the previous sections, singularities in the integrand can be transformed such that suitable combinations of Gauss-Legendre and Gauss-Jacobi quadrature can be employed. In the following we state the main result of \cite{chernov-schwab12} regarding numerical integration of certain singular integrals.
\begin{proposition}[\cite{chernov-schwab12}]
\label{2d_proposition}
Let $\mathcal{T}_{\gamma}$ be a $\gamma$-shape regular mesh and $S_{1}, S_{2} \in \mathcal{T}_{\gamma}$ be closed simplices in $\mathbb{R}^{d}$ with $k := \text{dim}(S_{1} \cap S_{2})$ (setting $k:=-1$ if $S_{1}\cap S_{2} = \emptyset$) and consider integrals of the form
	\begin{align}
		I = \int\limits_{S_{1}} \int\limits_{S_{2}} | \vec{x}-\vec{y} |^{\alpha} F(\vec{x},\vec{y},\vec{x}-\vec{y}) \: d\vec{y} \: d\vec{x},
	\end{align}
	where $\alpha \in \mathbb{R}$ and $F$ is a real analytic function, i.e., $F \in C^{\omega}(S_{1} \times S_{2} \times (S_{2}-S_{1}))$.

 Then, there exist $K_{k} \in \mathbb{N}$ depending only on $k$ and polynomial transformations $\Phi_{j}$, $j=0,\dots,K_k$ of degree 
$q_\Phi := \max_{j} \operatorname*{deg}(\Phi_j)$, depending only on $d$, such that  
 the integral $I$ takes the form
	\begin{align}\label{2d_trans_basis}
		I = \sum\limits_{j=0}^{K_{k}} \int\limits_{~[0,1]^{2d}} F\circ \Phi_{j}(\vec{\hspace{0.5mm}t}) \: \mathcal{R}_j(\vec{\hspace{0.5mm}t}) \: J_{\Phi_{j}}(\vec{\hspace{0.5mm}t}) \: t_{1}^{\alpha+2d-k-1} \,d\vec{\hspace{0.5mm}t},
	\end{align}
where $\mathcal{R}_j \in C^{\omega}([0,1]^{2d})$ are real analytic functions given by 

	\begin{align}
		\mathcal{R}_j(\vec{\hspace{0.5mm}t}) := \dfrac{ |(\vec{x}-\vec{y}) \circ \Phi_{j}(\vec{\hspace{0.5mm}t})|^{\alpha} }{t_{1}^{\alpha}}
	\end{align}
and the Jacobians $J_{\Phi_{j}}$ are polynomials of degree at most $d(q_\Phi-1)$.

	In particular, the condition $\alpha > k-2d$ ensures that \eqref{2d_trans_basis} is integrable and $t_{1}^{\alpha+2d-k-1}$ can be used as a Gauss-Jacobi weight function.
\end{proposition}
\begin{proof}
	See \cite[Sec.~3]{chernov-schwab12} for the explicit construction of the transformations $\Phi_j$ and the resulting polynomial degree $q_\Phi$ as well as \cite[Thm.~{4.1}, Rem.~{2}]{chernov-schwab12}, where a slightly different formulation is shown, which even includes the more general case that $F$ is in a Gevrey class. In \cite{chernov-schwab12} the condition $\alpha > k-2d$ is required, but it follows from inspection of
the proof that it is only needed to ensure integrability of the integrand. 
\end{proof}

\begin{remark}
\begin{enumerate}[label=(\roman*)]
\item The transformations $\Phi_j$ are, similarly to the case $d=1$, combinations of affine transformations and Duffy-like transformations that transform simplices to hypercubes and thus are polynomials. The 
 parameter $K_k \in \mathbb{N}$ accounts for different cases that have to be treated with different transformations (as can be seen in the case $d=1$ as well, compare (\ref{case_idpanel}) and (\ref{case_neighbored})). If $d>1$, this requires even more cases; however, structurally they are all similar, which allows for the compact notation.

\item An important observation of (\ref{2d_trans_basis}) is that the transformations (by employing relative coordinates) can be constructed such that the singularity of the function $|\vec{x}-\vec{y}|^{-d-2s}$ appears after transformation only in a single variable labelled $t_1$.

\item  Since the term $t_{1}^{\alpha+2d-k-1}$ with $ \alpha +2d-k-1 >-1 $ can be handled as a weight function with Gauss-Jacobi quadrature, an approximation to \eqref{2d_trans_basis} can be achieved by a tensor quadrature rule. 
\eremk
\end{enumerate}
\end{remark}

Unfortunately, the integrals in \eqref{2D_int} do not fulfill the requirement of the final statement of Proposition~\ref{2d_proposition} to be integrable since $\alpha = -d-2s > k-2d$ does not hold for all $0 \le k \le d$ and $s \in (0,1)$. Therefore, we have to modify the analysis of \cite{chernov-schwab12} to suit our integrand by showing that, after application of the transformations $\Phi_j$, 
the term $ (v(x)-v(y))(w(x)-w(y)) $ takes the form $t_{1}^{2} \: q(t_{1},\dots, t_{2d})$ where $q$ is a polynomial in $2d$ variables, i.e., $ q \in \mathcal{P}_k(\mathbb{R}^{2d}) $ for some $k$. 
Consequently, the singular term in the integral takes the form 
$t_{1}^{\tilde{\alpha}}$ with $\tilde{\alpha} := \alpha+2d-k+1 > -1$. 
More precisely, we have the following Corollary~\ref{corollar_2d}, which can be seen as an extension of \cite[Thm.~{4.1}]{chernov-schwab12} to the present specific case \eqref{2D_int}.

\begin{corollary}\label{corollar_2d}
Let $\mathcal{T}_{\gamma}$ be a $\gamma$-shape regular mesh and $S_{1}, S_{2} \in \mathcal{T}_{\gamma}$ be closed simplices in $\mathbb{R}^{d}$ with $k := \text{dim}(S_{1} \cap S_{2})$ (setting $k:=-1$ if $S_{1}\cap S_{2} = \emptyset$) and, for $v,w \in S_{0}^{p,1}(\mathcal{T}_{\gamma})$ consider the integral
	\begin{align}\label{corollar_2d_eq1}
		I_{S_{1},S_{2}}(v,w) := \int\limits_{S_{1}} \int\limits_{S_{2}} \dfrac{(v|_{S_{1}}(\vec{x})-v|_{S_{2}}(\vec{y}))(w|_{S_{1}}(\vec{x})-w|_{S_{2}}(\vec{y}))}{|\vec{x}-\vec{y}|^{d+2s}} \,d\vec{y}\, d\vec{x}.
	\end{align}
Then, employing, for $k \ge 0$, the polynomial transformations $\Phi_j$ of Proposition~\ref{2d_proposition} of degree (at most) $q_{\Phi}$ the integral $I_{S_{1},S_{2}}(v,w)$ takes the form
	\begin{align}\label{corollar_2d_eq2}
		I_{S_{1},S_{2}}(v,w) =  \int\limits_{~[0,1]^{2d}} \sum\limits_{j=0}^{K_{k}} \underset{=: F_{j}}{\underbrace{ P_{v,j}(\vec{\hspace{0.5mm}t}) \: P_{w,j}(\vec{\hspace{0.5mm}t}) \: \mathcal{R}_{j}(\vec{\hspace{0.5mm}t}) \: J_{\Phi_{j}}(\vec{\hspace{0.5mm}t)} }} \: t_{1}^{2-2s+d-k-1} \,d\vec{\hspace{0.5mm}t}.
	\end{align}
	Here, the Jacobians $J_{\Phi_{j}}$ are polynomials of degree (at most) $\leq d(q_\Phi-1)$, $R_j \in C^{\omega}([0,1]^{2d})$ are analytic functions given by 
	\begin{align}
		\mathcal{R}_{j}(\vec{\hspace{0.5mm}t}) := \dfrac{t_{1}^{d+2s}}{|(\vec{x}-\vec{y}) \circ \Phi_{j}(\vec{\hspace{0.5mm}t})|^{d+2s}},
	\end{align}
 and  $P_{v,j}, P_{w,j} \in \mathcal{P}(\mathbb{R}^{2d})$ are polynomials of degree (at most) $\leq p q_\Phi-1$, defined by
	\begin{align}\label{corollar_2d_poly_def}
		P_{v,j}(\vec{\hspace{0.5mm}t}) := \dfrac{(v|_{S_{1}}-v|_{S_{2}}) \circ \Phi_{j} (\vec{\hspace{0.5mm}t})}{t_{1}} ~~\text{ and }~~ P_{w,j}(\vec{\hspace{0.5mm}t}) := \dfrac{(w|_{S_{1}}-w|_{S_{2}}) \circ \Phi_{j} (\vec{\hspace{0.5mm}t})}{t_{1}}.
	\end{align}
	For $k = -1$, we get the form 
		\begin{align}\label{corollar_2d_eq2_getrennt}
		I_{S_{1},S_{2}}(v,w) =  \int\limits_{~[0,1]^{2d}} \underset{=: F_{-1}}{\underbrace{ P_{v,-1}(\vec{\hspace{0.5mm}t}) \: P_{w,-1}(\vec{\hspace{0.5mm}t})  \: \mathcal{R}_{-1}(\vec{\hspace{0.5mm}t}) \: J_{\Phi_{-1}}(\vec{\hspace{0.5mm}t}) }}  \,d\vec{\hspace{0.5mm}t},
	\end{align}
with polynomial Jacobian $J_{\Phi_{-1}}$, $\mathcal{R}_{-1}(\vec{\hspace{0.5mm}t}) := |(\vec{x}-\vec{y}) \circ \Phi_{-1}(\vec{\hspace{0.5mm}t})|^{-d-2s}$ analytic  and 
polynomials $P_{v,-1}, P_{w,-1} \in \mathcal{P}_{p q_\Phi-1}(\mathbb{R}^{2d})$ defined by
	\begin{align}\label{corollar_2d_poly_def_getrennt}
		P_{v,-1}(\vec{\hspace{0.5mm}t}) := (v|_{S_{1}}-v|_{S_{2}}) \circ \Phi_{-1} (\vec{\hspace{0.5mm}t}) ~~\text{ and }~~ P_{w,-1}(\vec{\hspace{0.5mm}t}) := (w|_{S_{1}}-w|_{S_{2}}) \circ \Phi_{-1} (\vec{\hspace{0.5mm}t}).
	\end{align}
\end{corollary}
\begin{proof}
	For $k \ge 0$, with Proposition~\ref{2d_proposition} it is only left to show that $P_{v,j}$ and $P_{w,j}$ from \eqref{corollar_2d_poly_def} are polynomials. We only prove the statement for $P_{v,j}$.
	
	Since $v \in S_{0}^{p,1}(\mathcal{T}_{\gamma})$ is a piecewise continuous polynomial, the singularity points $\vec{x} = \vec{y}$ of $|\vec{x}-\vec{y}|^{-d-2s}$ are a subset of the roots of the polynomial $(v|_{S_{1}}(\vec{x})-v|_{S_{2}}(\vec{y}))$. Since $\Phi_{j}$ 
is a polynomial, it follows that $(v|_{S_{1}}-v|_{S_{2}}) \circ \Phi_{j}$ is also a polynomial (of degree bounded by $pq_\Phi$) that vanishes at the singularities of $ |\vec{x}-\vec{y}|^{-d-2s} \circ \Phi_{j} $. So the separated singularity $t_{1}$ has to be a root of $(v|_{S_{1}}-v|_{S_{2}}) \circ \Phi_{j}$. The fundamental theorem of algebra finishes the proof.
	
	For $k = -1$ the proof follows immediately from Step~1 and 2 of the transformations of \cite[Sec.~{3}]{chernov-schwab12}.
\end{proof}
\cite[Thm.~{5.4}]{chernov-schwab12} also asserts exponential convergence of a suitable combination of Gauss-Jacobi and Gauss-Legendre quadrature employed to integrands covered by Proposition~\ref{2d_proposition}.

\begin{proposition}\label{2d_proposition2}
	Let $\mathcal{R} \in C^{\omega}([0,1]^{d'})$ and $\beta_{1} > -1$. Then, there exist $C$, $b>0$ independent of $d'$ 
such that for all $n\in \mathbb{N}$ there holds 
	\begin{align}\label{2d_proposition2_formel}
		\bigg| \int_{[0,1]^{d'}} t_{1}^{\beta_{1}} \mathcal{R}(\vec{\hspace{0.5mm} t }) \, d\vec{\hspace{0.5mm} t } - GJ_{n,t_{1}}^{0,\beta_{1}} \circ GL_{n,t_{2}}\circ \cdots \circ GL_{n,t_{d'}}(\mathcal{R}) \bigg| \le C \exp(-b \: N^{1/d'}),
	\end{align}
	where $N = \mathcal{O}(n^{d'})$ is the total number of quadrature points.
\end{proposition}

Propositions~\ref{2d_proposition} and \ref{2d_proposition2} are formulated for fairly general integrands. 
However, in order to obtain exponential convergence results for $hp$-FEM discretizations, as in the case $d=1$, an explicit dependence of the convergence rate on the employed polynomial degree has to be derived, which is not directly deducible from Proposition~\ref{2d_proposition2}. 

In the following we extend our $1d$-quadrature analysis, which was explicit in $p$, to higher dimension $d>1$ specifically for the easier case of $\gamma$-shape regular meshes $\mathcal{T}_\gamma$ with a finite number of patch configurations.

We will
make the following assumption on the structure of the underlying triangulation of $\Omega$: 
	\begin{assumption}
		\label{assumption:finitely-many-patches}
		The triangulation ${\mathcal T}_\gamma$ is $\gamma$-shape regular and there exists, 
		up to dilations, rotations, and translations a finite number (independent on the number of elements in the mesh) of different 
		patches (i.e., unions of elements sharing a vertex). This is, for example, ensured for $d \in \{2,3\}$, if the mesh is generated from a coarse mesh
		by ``newest vertex bisection'', \cite{karkulik-pavlicek-praetorius13,stevenson08}.
	\end{assumption}

\begin{remark}
For exponential convergence results in terms 
of ``error vs.\ number of degrees of freedom'' as in  
Proposition~\ref{theorem_galerkin} or Theorem~\ref{main_result}, special geometric meshes $\mathcal{T}_{geo}$ are required that 
include anisotropic elements, \cite{faustmann2023hp}. A quadrature analysis on such meshes requires a more careful analysis of elements with large aspect ratio and is postponed to a forthcoming work. 
\eremk
\end{remark}
\subsection{Consistency error analysis}
	We start with a standard quadrature rule on a simplex $S$. To that end, we can also use the affine transformation \cite[Sec.~3 (Step 1)]{chernov-schwab12} to map a given simplex $S$ to the reference simplex $\widehat{S}_{d}:= \{ (x_{1},\dots, x_{d}) ~|~ x_i\ge 0 \;\forall i=1,\dots,d, \: x_{1}+\cdots + x_{d} \le 1 \}$ and  afterwards with the Duffy type transformation \cite[(2.12)]{chernov-schwab12} to $[0,1]^{d}$. This then allows to use tensor product Gauss-Legendre rules to obtain
	\begin{align}\label{Gauss_Simplex}
		\int_{S} f(\vec{x}) \, d\vec{x} = \int_{[0,1]^d} f \circ \Phi_{S}(\vec{\hspace{0.5mm}t}) \, J_{\Phi_{S}}(\vec{\hspace{0.5mm}t}) \, d\vec{\hspace{0.5mm}t}  \approx GL_{n,t_{1}}\circ \cdots \circ GL_{n,t_{d}}(f \circ \Phi_{S} \, J_{\Phi_{S}})  =: GL_{S}^{n}(f),
	\end{align}
	where $\Phi_{S}$ denotes the composed polynomial transformations \cite[Sec.~3 (Step 1) with (2.12)]{chernov-schwab12} depending only on the simplex $S$ with its polynomial Jacobian $J_{\Phi_{S}}$. Since $\Phi_{S}$ is an affine transformation composed with a Duffy type transformation, it holds for polynomials $u \in \mathcal{P}_{p}(S)$ that $u \circ \Phi_{S} \in \mathcal{Q}_{p}({\mathbb R}^d)$.

	The approximation of the right-hand side $l(v) := \skp{f,v}_{L^{2}(\Omega)}$ follows immediately.

		\begin{definition}[Approximate linear form for $d > 1$] \label{lin_multi_d_approx}
			For a piecewise polynomial $v \in S_{0}^{p,1}(\mathcal{T}_{\gamma})$, we define the approximate linear form by 
			\begin{align}\label{functional_aprox_definition_multi}
				l(v) := \skp{f,v}_{L^{2}(\Omega)} = \sum_{S \in \mathcal{T}_{\gamma}} \int_S f(\vec{x}) v(\vec{x}) \; d\vec{x} 
				 \approx \sum_{S \in \mathcal{T}_{\gamma}} GL_{S}^{n} \big( f \, v \big) =:\widetilde{l}_{n}(v),
			\end{align}
			where $GL^n_{S}$ denotes the tensor product Gauss-Legendre rule \eqref{Gauss_Simplex}.
		\end{definition}

Consistency error estimates for the linear form $l$  follows with the same arguments as for the one dimensional case in Lemma~\ref{consistency_error_lemma_functional}\revision{.}

	\begin{lemma}[Consistency error for $l$]
		\label{consistency_error_lemma_functional_multi}
		Let $f$ be analytic in $\overline{\Omega}$, and let $\mathcal{T}_{\gamma}$ be a $\gamma$-shape regular mesh on $\Omega \subseteq \mathbb{R}^{d}$. Let $ l(v) := \skp{f,v}_{L^{2}(\Omega)}$ and let its approximation
		$ \widetilde{l}_{n}(\cdot)$ be defined by \eqref{functional_aprox_definition_multi}.
		Then, there exist constants $\rho > 1$ and $C_{f,\gamma,s} > 0$ depending only on $f$, $\gamma$, $s$, and $\Omega$ such that
		\begin{align}
			|l(v)-\widetilde{l}_{n}(v)| \le C_{f,\gamma,s} \,  p \;   \rho^{p-2n+1} \|v\|_{\widetilde{H}^{s}(\Omega)} &~~~~~~\text{for all } ~ v \in S_{0}^{p,1}(\mathcal{T}_{\gamma}).
		\end{align}
	\end{lemma}

Next, we define the approximation to the bilinear form $a(\cdot,\cdot)$.

\begin{definition}[Approximate bilinear form for $d>1$] \label{bilin_multi_d_approx}
Let $\mathcal{T}_{\gamma}$ be a $\gamma$-shape regular mesh and $S_{1}, S_{2} \in \mathcal{T}_{\gamma}$ be closed simplices in $\mathbb{R}^{d}$ with $k := \text{dim}(S_{1} \cap S_{2})$ (setting $k:=-1$ if $S_{1}\cap S_{2} = \emptyset$). For piecewise polynomials $v \in S_{0}^{p,1}(\mathcal{T}_{\gamma})$, $w \in S_{0}^{r,1}(\mathcal{T}_{\gamma})$, using the notations $F_j$, $j=-1,\dots,K_k$ from Corollary~\ref{corollar_2d}, we define the following tensor product quadrature rules
	\begin{align}
		Q^{n}_{S_{1},S_{2}}(v,w)&:= GJ_{n,t_{1}}^{0,\beta_{1}} \circ GL_{n,t_{2}}\circ \cdots \circ GL_{n,t_{d}} \big( \textstyle \sum_{j=0}^{K_{k}} F_{j} \big) &&~~\text{for}~~ k\ge0, \\
		Q^{n}_{S_{1},S_{2}}(v,w)&:= GL_{n,t_{1}}\circ \cdots \circ GL_{n,t_{d}} \big(  F_{-1} \big)&&~~\text{for}~~ k=-1,
	\end{align}
	where $\beta_{1}:= 1-2s+d-k$. 

The final approximation to the bilinear form $a(\cdot,\cdot)$ reads 
	\begin{align}\label{bilinear_aprox_definition_mehrd}
		\widetilde{a}_{n}(v,w):= \frac{C(s,d)}{2}  \sum\limits_{S_{1} \in \mathcal{T}_{\gamma}} \sum\limits_{S_{2} \in \mathcal{T}_{\gamma}} Q_{S_{1},S_{2}}^n(v,w) + \sum\limits_{S_{1} \in \mathcal{T}_{\gamma}} Q^n_{S_{1},\Omega^{c}}(v,w).
	\end{align}
Here $Q^n_{S_{1},\Omega^{c}}(v,w)$ denotes an approximation to $I_{S_{1},\Omega^c}(v,w)$ given by (\ref{eq:Qnomegac}). 
\end{definition}

We now employ scaling arguments to work out the dependence on the element sizes and the polynomial degree when estimating $a(\cdot,\cdot) - \widetilde{a}_n(\cdot,\cdot)$. 

\subsubsection*{Adjacent or identical simplices}

We start with the case  of two simplices $S_1,S_2$ with $ k := \text{dim}(S_{1} \cap S_{2}) \ge 0 $.
We define the reference simplex as $\widehat{S}_{d}:= \{ (x_{1},\dots, x_{d}) ~|~ x_i\ge 0 \;\forall i=1,\dots,d, \: x_{1}+\cdots + x_{d} \le 1 \}$. 
As the simplices $S_i$, $i=1,2$ share, by assumption, $k+1$ vertices, we may label the vertices $\vec{\hspace{0.3mm} v }^{(i,\ell)}$ of $S_i$ such that
 $ \vec{\hspace{0.3mm} v }^{(1,\ell)} =  \vec{\hspace{0.3mm} v }^{(2,\ell)}$ for all $0 \le \ell \le k$ and $ \vec{\hspace{0.3mm} v }^{(1,\ell)} \neq  \vec{\hspace{0.3mm} v }^{(2,\ell)}$ for all $k+1 \le \ell \le d$. With the $d\times d\:$-$\:$matrices
\begin{align}
	A^{(i)}:= \big(\vec{\hspace{0.3mm} v }^{(i,1)} - \vec{\hspace{0.3mm} v }^{(i,0)} ~ \cdots ~ \vec{\hspace{0.3mm} v }^{(i,d)} - \vec{\hspace{0.3mm} v }^{(i,0)} \big), ~~~ i=1,2,
\end{align}
the pullback transformation $F_{S_{1}\times S_{2}}$ is given by
\begin{align}
	F_{S_{1}\times S_{2}}: \widehat{S}_{d}\times\widehat{S}_{d} \rightarrow  S_{1}\times S_{2}, ~~~ (\vec{x},\vec{y}) \mapsto \big( F_{S_{1}} (\vec{x}) , F_{S_{2}} (\vec{y})  \big) := \Big(\vec{\hspace{0.3mm} v }^{(1,0)} + A^{(1)} \vec{x},\: \vec{\hspace{0.3mm} v }^{(2,0)} + A^{(2)} \vec{y} \Big)
\end{align}
with its Jacobian $ J_{F_{S_{1} \times S_{2}}} = |\det A^{(1)} \: \det A^{(2)}| $. Denoting by $\hat{v}_{S_{i}} := v|_{S_{i}} \circ F_{S_{i}}$ and $ \hat{w}_{S_{i}} := w|_{S_{i}} \circ F_{S_{i}} $ for $i=1,2$ the pullbacks to the reference simplex $\widehat{S}_{d}$, $F_{S_{1}\times S_{2}}$ transforms the integral \eqref{corollar_2d_eq1}
to
\begin{align}\label{int_2d_ref}
	I_{S_{1},S_{2}}(v,w) = \int\limits_{\widehat{S}_{d}} \int\limits_{\widehat{S}_{d}} \dfrac{(\hat{v}_{S_{1}} (\vec{x}) - \hat{v}_{S_{2}} (\vec{y})) \: (\hat{w}_{S_{1}} (\vec{x})- \hat{w}_{S_{2}} (\vec{y}))}{| F_{S_{1}}(\vec{x}) - F_{S_{2}} (\vec{y})|^{d+2s}} J_{F_{S_{1} \times S_{2}}} \,d\vec{y}\, d\vec{x}.
\end{align}
As, for all elements in a $\gamma$-shape regular mesh $\mathcal{T}_{\gamma}$, the lengths of all edges $|\vec{\hspace{0.3mm} v }^{(j,i)} - \vec{\hspace{0.3mm} v }^{(j,0)}|$ are controlled by the element diameter $h_{S_j}$, we obtain $ J_{F_{S_{1} \times S_{2}}} \le C_{\gamma,d} \: h_{S_{1}}^{d} \: h_{S_{2}}^{d} $ with a constant $C_{\gamma,d}$ that depends only on $\gamma$ and the dimension $d$.


To simplify the notation we introduce $\mathcal{N}_{\hat{v}}(\vec{x}, \vec{y}):= \hat{v}_{S_{1}} (\vec{x}) - \hat{v}_{S_{2}} (\vec{y}) $ and  $\mathcal{N}_{\hat{w}}(\vec{x}, \vec{y}):= \hat{w}_{S_{1}} (\vec{x}) - \hat{w}_{S_{2}} (\vec{y}) $.
Corollary \ref{corollar_2d} yields for \eqref{int_2d_ref} 
\begin{align}\label{2d_integrand_absch_start}
	\hspace{-6mm} I_{S_{1},S_{2}}(v,w) = J_{F_{S_{1}  \times S_{2}}} \hspace{-3mm} \int\limits_{~[0,1]^{2d}}   \underbrace{\sum\limits_{j=0}^{K_{k}} \dfrac{\mathcal{N}_{\hat{v}} \circ \Phi_{j}(\vec{\hspace{0.5mm}t})}{t_{1}} \: \dfrac{\mathcal{N}_{\hat{w}} \circ \Phi_{j}(\vec{\hspace{0.5mm}t})}{t_{1}} \: \dfrac{t_{1}^{d+2s}}{|(A^{(1)}\vec{x}-A^{(2)}\vec{y}) \circ \Phi_{j}(\vec{\hspace{0.5mm}t})|^{d+2s}} \: J_{\Phi_{j}}(\vec{\hspace{0.5mm}t})}_{=: {\mathcal I(\vec{\hspace{0.5mm}t})}}\:  t_{1}^{1-2s+d-k} \,d\vec{\hspace{0.5mm}t}.
\end{align} 
The estimate of the consistency error is again based on Lemma~\ref{lemma_quadabsch}, which directly generalizes to higher dimensions. Corollary~\ref{corollar_2d} shows that $\mathcal{I}$ allows for a holomorphic extension to a Bernstein ellipse $\widehat{\mathcal{E}_{\rho}}$ in each variable with fixed $\rho_{A^{(1)},A^{(2)}} > 1$, ostensibly dependent on the transformation matrices $A^{(1)},A^{(2)}$ but  independent of $v \in S_{0}^{p,1}(\mathcal{T}_{\gamma})$, $w \in S_{0}^{r,1}(\mathcal{T}_{\gamma})$. By Assumption~\ref{assumption:finitely-many-patches}, there is only a finite number of patch configurations in $\mathcal{T}_\gamma$, which leads, up to scaling, to a finite number of different matrices $ A^{(1)},A^{(2)}$. To remove the scaling dependence, we note that
\begin{align}\label{scaling_argument}
	A^{(1)}\vec{x} - A^{(2)}\vec{y} = h_{S_{1}} \Big( h_{S_{1}}^{-1} \: A^{(1)}\vec{x} - h_{S_{1}}^{-1} \: A^{(2)}\vec{y} \Big).
\end{align}
For $\gamma$-shape regular meshes we have $h_{S_{1}} \sim h_{S_{2}}$ and the diameter of each simplex is proportional to all edge lengths, which leads for $ \widehat{A}^{i} := h_{S_{1}}^{-1} \: A^{(i)} $ to $ \| \widehat{A}^{i} \|_{1} = \mathcal{O}(1) $ for $i=1,2$ and subsequently to a finite number of different  values $\rho_{\widehat{A}^{(1)},\widehat{A}^{(2)}} > 1$. Thus, we have a holomorphic extension of \revision{${\mathcal I}$} to a Bernstein ellipse $\widehat{\mathcal{E}}_{\rho} $ with a fixed $\rho := \min_{\widehat{A}^{(1)},\widehat{A}^{(2)}}(\rho_{\widehat{A}^{(1)},\widehat{A}^{(2)}}) > 1$. To finish the estimate of the consistency error, it suffices to bound each of the three quotients in \eqref{2d_integrand_absch_start} in the norms $\| \cdot \|_{L^{\infty} (\widehat{\mathcal{I}}^{2d \setminus \ell} \times \widehat{\mathcal{E}}_{\rho}^{\ell}) }$, where $\widehat{\mathcal{I}} := (0,1)$ and $ \widehat{\mathcal{I}}^{2d \setminus \ell} \times \widehat{\mathcal{E}}_{\rho}^{\ell}:= \widehat{\mathcal{I}} \times \cdots \times \widehat{\mathcal{I}} \times \widehat{\mathcal{E}}_{\rho}^{\ell} \times \widehat{\mathcal{I}} \times \cdots \times \widehat{\mathcal{I}} $ denotes the set where the $\ell$-th component of $\widehat{\mathcal{I}}^{2d}$ is extended to the Bernstein ellipse $\widehat{\mathcal{E}}_{\rho}$.

Using  $ \hat{v}_{S_{1}}(\vec{0}) = \hat{v}_{S_{2}}(\vec{0}) $ for $k \ge 0$, the first term can be bounded as in Lemma~\ref{lemma_norm_approx} using the 
Bernstein and Markov  inequalities by 
\begin{align}
	\big\| &\mathcal{N}_{\hat{v}}  \circ \Phi_j \: t_{1}^{-1} \big\|_{L^{\infty} (\widehat{\mathcal{I}}^{2d \setminus \ell} \times \widehat{\mathcal{E}}_{\rho}^{\ell}) } = 
	\bigg\| \int_{0}^{t_{1}} \partial_{\tau} \big( \mathcal{N}_{\hat{v}} \circ \Phi_j(\tau,t_2,\dots,t_d) \big) \: d\tau \: t_{1}^{-1} \bigg\|_{L^{\infty} (\widehat{\mathcal{I}}^{2d \setminus \ell} \times \widehat{\mathcal{E}}_{\rho}^{\ell}) } \notag \\
	&\le \big\| \partial_{t_{1}} \big( \mathcal{N}_{\hat{v}} \circ \Phi_j(\vec{\hspace{0.5mm} t }) \big) \big\|_{L^{\infty} (\widehat{\mathcal{I}}^{2d \setminus \ell} \times \widehat{\mathcal{E}}_{\rho}^{\ell}) } 
	\le \rho^{q_\Phi p} \big\| \partial_{t_{1}} \big( \mathcal{N}_{\hat{v}} \circ \Phi_j(\vec{\hspace{0.5mm} t }) \big) \big\|_{L^{\infty} (\widehat{\mathcal{I}}^{2d}) } \notag \\
	&\lesssim  (q_\Phi p)^{2} \rho^{q_\Phi p} \big\|  \mathcal{N}_{\hat{v}} \circ \Phi_j(\vec{\hspace{0.5mm} t })  \big\|_{L^{\infty} (\widehat{\mathcal{I}}^{2d}) } 
	=  (q_\Phi p)^{2} \rho^{q_\Phi p} \big\|  \mathcal{N}_{\hat{v}}   \big\|_{L^{\infty} (  \Phi_{j}(\widehat{\mathcal{I}}^{2d})) } 
	\le (q_\Phi p)^{2} \rho^{q_\Phi p} \big\|  \mathcal{N}_{\hat{v}}   \big\|_{L^{\infty} ( \widehat{S}_{d} \times \widehat{S}_{d}) }  \notag \\
	&= (q_\Phi p)^{2} \rho^{q_\Phi p} \big\|  \hat{v}_{S_{1}} - \hat{v}_{S_{1}}(\vec{0}) + \hat{v}_{S_{2}}(\vec{0}) - \hat{v}_{S_{2}}   \big\|_{L^{\infty} ( \widehat{S}_{d} \times \widehat{S}_{d}) } \notag \\
	&\le  (q_\Phi p)^{2} \rho^{ q_\Phi p} \big( \big\|  \hat{v}_{S_{1}} - \hat{v}_{S_{1}}(\vec{0})  \big\|_{L^{\infty} ( \widehat{S}_{d}) } + \big\|  \hat{v}_{S_{2}} - \hat{v}_{S_{2}}(\vec{0})   \big\|_{L^{\infty} ( \widehat{S}_{d}) }  \big),
\end{align}
where, again, $q_\Phi$ is the maximal degree of the polynomial transformations $\Phi_j$. On the reference simplex, there holds by Markov's inequality and inductive application of the inverse inequality from \cite[Thm.~3.92]{schwab1998p} that
\begin{align}
	\big\|  \hat{v}_{S_{1}} - \hat{v}_{S_{1}}(\vec{0})  \big\|_{L^{\infty} ( \widehat{S}_{d}) } 
	\lesssim \big\|  \nabla \hat{v}_{S_{1}}  \big\|_{L^{\infty} ( \widehat{S}_{d}) } 
	= \big\|  \nabla (\hat{v}_{S_{1}} - \overline{\hat{v}_{S_{1}}} )  \big\|_{L^{\infty} ( \widehat{S}_{d}) } 
	\lesssim (q_\Phi p)^{2} \big\|  \hat{v}_{S_{1}} - \overline{\hat{v}_{S_{1}}}   \big\|_{ L^{\infty} ( \widehat{S}_{d}) } \notag \\
	\lesssim (q_\Phi p)^{2+d} \big\|  \hat{v}_{S_{1}} - \overline{\hat{v}_{S_{1}}}   \big\|_{ L^{2} ( \widehat{S}_{d}) }
	\lesssim (q_\Phi p)^{2+d} \big|  \hat{v}_{S_{1}}   \big|_{ H^{s} ( \widehat{S}_{d}) } 
	\le C_{\gamma,d,s} (q_\Phi p)^{2+d} h_{S_{1}}^{s-d/2} \big| \hspace{0.3mm}  v|_{S_{1}}   \big|_{ H^{s} ( S_{1}) } ,
\end{align}
	where $C_{\gamma,d,s}$ is a constant that depends only on $\gamma,d,s$. This finishes the upper bound for the first quotient in (\ref{2d_integrand_absch_start})
	\begin{align}
		\big\| \mathcal{N}_{\hat{v}}  \circ \Phi_j \: t_{1}^{-1} \big\|_{L^{\infty} (\widehat{\mathcal{I}}^{2d \setminus \ell} \times \widehat{\mathcal{E}}_{\rho}^{\ell}) } 
		&\le C_{\gamma,d,s} \: \rho^{q_\Phi p} \: (q_\Phi p)^{4+d} \max \Big( h_{S_{1}}^{s-d/2} \: \big| \hspace{0.3mm}  v|_{S_{1}}   \big|_{ H^{s} ( S_{1}) }, \: h_{S_{2}}^{s-d/2} \: \big| \hspace{0.3mm}  v|_{S_{2}}   \big|_{ H^{s} ( S_{2}) } \Big) \notag \\
		&\le C_{\gamma,d,s} \: \rho^{q_\Phi p} \: (q_\Phi p)^{4+d}  h_{S_{1}}^{s-d/2} \: \big|   v   \big|_{ H^{s} ( co(S_{1},S_{2})) }.
	\end{align}

	The second factor in the integrand in (\ref{2d_integrand_absch_start}) can be treated in the same way. The estimate for the third factor in the integrand follows again, as discussed above, by Assumption~\ref{assumption:finitely-many-patches} and   \eqref{scaling_argument}
	\begin{align*}
	&\bigg\| \dfrac{t_{1}^{d+2s}}{|(A^{(1)}\vec{x}-A^{(2)}\vec{y}) \circ \Phi_{j}(\vec{\hspace{0.5mm}t})|^{d+2s}} \bigg\|_{L^{\infty} (\widehat{\mathcal{I}}^{2d \setminus \ell} \times \widehat{\mathcal{E}}_{\rho}^{\ell}) } \\ 
&\qquad\qquad=
h_{S_{1}}^{-d-2s} \Bigg\| \dfrac{t_{1}^{d+2s}}{\big|\big( h_{S_{1}}^{-1}  A^{(1)}\vec{x} - h_{S_{1}}^{-1}  A^{(2)}\vec{y} \big) \circ \Phi_{j}(\vec{\hspace{0.5mm}t}) \big|^{d+2s}} \Bigg\|_{L^{\infty} (\widehat{\mathcal{I}}^{2d \setminus \ell} \times \widehat{\mathcal{E}}_{\rho}^{\ell}) } 
\leq  C_{\gamma,s,\revision{\rho},d} \: h_{S_{1}}^{-d-2s},
	\end{align*}
where the last estimate follows from the observation that we only have  a finite number of cases for the function inside the norm. 
Now, we have deduced the appropriate scaling in terms of the element sizes for each factor in (\ref{2d_integrand_absch_start}) in the $L^\infty$-norm  and
inserting everything into the higher-dimensional analog of Lemma~\ref{lemma_quadabsch} yields 
\begin{align}
	\hspace{-5mm}|I_{S_{1},S_{2}}(v,w) - Q^{n}_{S_{1},S_{2}}(v,w)| \le C_{s,\gamma,\rho,d} \: (q_{\Phi}p)^{d+4} (q_{\Phi}r)^{d+4}  \rho^{q_\Phi(p+r)-2n+1} \big|   v   \big|_{ H^{s} ( co(S_{1},S_{2})) } \big|   w   \big|_{ H^{s} ( co(S_{1},S_{2})) }
\end{align}
for adjacent or identical simplices $S_{1},S_{2}$.

\subsubsection*{Separated simplices}

	For the case $ S_{1} \cap S_{2} = \emptyset$, i.e. $k = -1$, we start with the same transformation as in \eqref{int_2d_ref}, where we labelled the vertices such that there holds $ \vec{\hspace{0.3mm} v }^{(1,0)} - \vec{\hspace{0.3mm} v }^{(2,0)} = \min_{i,j} \vec{\hspace{0.3mm} v }^{(1,i)} - \vec{\hspace{0.3mm} v }^{(2,j)} $. Corollary~\ref{corollar_2d} yields for \eqref{int_2d_ref} 
	\begin{align}\label{2d_integrand_absch_start_getrennt}
		I_{S_{1},S_{2}}(v,w) = J_{F_{S_{1}  \times S_{2}}} \hspace{-3mm} \int\limits_{~[0,1]^{2d}} \mathcal{N}_{\hat{v}} \circ \Phi_{-1}(\vec{\hspace{0.5mm}t}) \: \mathcal{N}_{\hat{w}} \circ \Phi_{-1}(\vec{\hspace{0.5mm}t}) \: |(F_{S_{1}}(\vec{x}) - F_{S_{2}} (\vec{y})) \circ \Phi_{-1}(\vec{\hspace{0.5mm}t})|^{-d-2s} \: J_{\Phi_{-1}}(\vec{\hspace{0.5mm}t)} \,d\vec{\hspace{0.5mm}t}. 
	\end{align}
For simplices $S_1$, $S_2$ define $d_{S_{1},S_{2}} := \operatorname{dist}(S_{1},S_{2})$ and pick a closed ball $B_{S_1,S_2}$ with $S_1$, $S_2 \subseteq B_{S_1,S_2}$ 
and $\operatorname{diam} B_{S_1,S_2} \leq h_{S_1} + h_{S_2} + d_{S_{1},S_{2}}$.

The integrand can be estimated with a combination of arguments applied to the case $k \ge 0$ and the case $d=1$ in Lemma~\ref{lemma_quadrature_error}. Inserting the mean 
$\overline{v_{B_{S_{1},S_{2}}}} := \int_{B_{S_{1},S_{2}}} v(x) dx / |B_{S_{1},S_{2}}|$ gives
	\begin{align}
		\big\| &\mathcal{N}_{\hat{v}}  \circ \Phi_{-1}  \big\|_{L^{\infty} (\widehat{\mathcal{I}}^{2d \setminus \ell} \times \widehat{\mathcal{E}}_{\rho}^{\ell}) } \le 
		\rho^{q_\Phi p }  \bigg\|\mathcal{N}_{\hat{v}} \circ \Phi_{-1} \bigg\|_{L^{\infty} (\widehat{\mathcal{I}}^{2d}) }
		\le C\rho^{q_\Phi p }  \bigg\| \hat{v}_{S_{1}}  - \hat{v}_{S_{2}}  \bigg\|_{L^{\infty} (\widehat{S}_{d} \times \widehat{S}_{d}) } \notag \\
		&\le C  \rho^{q_\Phi p } \big( \big\|  \hat{v}_{S_{1}}  - \overline{v_{B_{S_{1},S_{2}}}}  \big\|_{L^{\infty} ( \widehat{S}_{d}) } + \big\|   \hat{v}_{S_{2}} - \overline{v_{B_{S_{1},S_{2}}}}  \big\|_{L^{\infty} ( \widehat{S}_{d}) }  \big).
	\end{align}
With an $L^\infty$--$L^2$ inverse estimate on the reference simplex and a Poincar\'e type estimate  for the ball $B_{S_1,S_2}$ 
there holds 
	\begin{align}
		\big\|  \hat{v}_{S_{1}}  - \overline{v_{B_{S_{1},S_{2}}}}  \big\|_{ L^{\infty} ( \widehat{S}_{d}) }
		&\le  C_d
		(q_\Phi p)^{d} \big\|  \hat{v}_{S_{1}}  - \overline{v_{B_{S_{1},S_{2}}}}   \big\|_{ L^{2} ( \widehat{S}_{d}) }
		\le C_{\gamma,d,s} \, h_{S_{1}}^{-d/2}
		(q_\Phi p)^{d}  \big\|  v  - \overline{v_{B_{S_{1},S_{2}}}}   \big\|_{ L^{2} ( S_{1}) } \notag \\
		&\le C_{\gamma,d,s} \,  h_{S_{1}}^{d/2} (h_{S_1} + h_{S_2} + d_{S_{1},S_{2}})^s
		(q_\Phi p)^{d}   \big|  v  \big|_{ H^{s} (B_{S_{1},S_{2}}) },
	\end{align}
	where $C_{\gamma,d,s}$ is a constant that depends only on $\gamma,d,s$. For the third factor in the integrand in (\ref{2d_integrand_absch_start_getrennt}), we note 
	\begin{align}
		 |(F_{S_{1}}(\vec{x}) - F_{S_{2}} (\vec{y})) \circ \Phi_{-1}(\vec{\hspace{0.5mm}t})|^{-d-2s}  
		 =
		|(\vec{\hspace{0.3mm} v }^{(1,0)} - \vec{\hspace{0.3mm} v }^{(2,0)} + A^{(1)}\vec{x}-A^{(2)}\vec{y}) \circ \Phi_{-1}(\vec{\hspace{0.5mm}t})|^{-d-2s}.
	\end{align}
	It follows that
	\begin{align}
		\big\| &|(F_{S_{1}}(\vec{x}) - F_{S_{2}} (\vec{y})) \circ \Phi_{-1}(\vec{\hspace{0.5mm}t})|^{-d-2s} \big\|_{L^{\infty} (\widehat{\mathcal{I}}^{2d \setminus \ell} \times \widehat{\mathcal{E}}_{\rho}^{\ell}) }  \notag \\
		&\le d_{S_{1},S_{2}}^{-d-2s}
		\big\| |(d_{S_{1},S_{2}}^{-1} (\vec{\hspace{0.3mm} v }^{(1,0)} - \vec{\hspace{0.3mm} v }^{(2,0)})  +d_{S_{1},S_{2}}^{-1} A^{(1)}\vec{x}- d_{S_{1},S_{2}}^{-1} A^{(2)}\vec{y}) \circ \Phi_{-1}(\vec{\hspace{0.5mm}t})|^{-d-2s} \big\|_{L^{\infty} (\widehat{\mathcal{I}}^{2d \setminus \ell} \times \widehat{\mathcal{E}}_{\rho}^{\ell}) }.
	\end{align}
	By Assumption~\ref{assumption:finitely-many-patches}, there is only a finite number of patch configurations in $\mathcal{T}_\gamma$, which leads, up to scaling, to a finite number of different matrices $A^{(1)},A^{(2)}$.
The $\gamma$-shape regularity and choice of numbering of the vertices yield  $ | (\vec{\hspace{0.3mm} v }^{(1,0)} - \vec{\hspace{0.3mm} v }^{(2,0)})| \sim d_{S_{1},S_{2}} $ and $\| d_{S_{1},S_{2}}^{-1} A^{(1)} \|, \, \| d_{S_{1},S_{2}}^{-1} A^{(2)} \|   =  \mathcal{O}(1)$. This leads to a finite number of holomorphic extensions. Hence,  
  there is a $\rho > 1$ for which a holomorphic extension in each variable to the Bernstein ellipse $\widehat{\mathcal{E}}_{\rho}$ is possible, and this extension 
can be bounded by 
	\begin{align}
		\big\| &|(F_{S_{1}}(\vec{x}) - F_{S_{2}} (\vec{y})) \circ \Phi_{-1}(\vec{\hspace{0.5mm}t})|^{-d-2s} \big\|_{L^{\infty} (\widehat{\mathcal{I}}^{2d \setminus \ell} \times \widehat{\mathcal{E}}_{\rho}^{\ell}) } \le C_{\gamma,s,\rho,d} \, d_{S_{1},S_{2}}^{-d-2s}. 
	\end{align}
Inserting everything into the higher-dimensional analog of Lemma~\ref{lemma_quadabsch} yields for separated simplices $S_{1},S_{2}$
\begin{align}
	|I_{S_{1},S_{2}}(v,w) - Q^{n}_{S_{1},S_{2}}(v,w)| \le C_{s,\gamma,d} \: (q_{\Phi}p)^{d} (q_{\Phi}r)^{d}  \rho^{q_\Phi(p+r)-2n+1} \big|   v   \big|_{ H^{s} ( B_{S_{1},S_{2}}) } \big|   w   \big|_{ H^{s} ( B_{S_{1},S_{2}}) },
\end{align}
where we used that for $\gamma$-shape regular meshes there holds $d_{S_{1},S_{2}} \geq C \max\{h_{S_1},h_{S_2}\}$ for some $C > 0$ depending on $\gamma$ so that 
the combined effect of the scaling parameters of all contributions in (\ref{2d_integrand_absch_start_getrennt}) can be uniformly bounded by
	\begin{align*}
		d_{S_{1},S_{2}}^{-d-2s} \: h_{S_{1}}^{d} \: h_{S_{2}}^{d} (h_{S_1} + h_{S_2} + d_{S_{1},S_{2}})^{2s} \big(h_{S_{1}}^{-d/2} + h_{S_{2}}^{-d/2}\big)^{2} \leq C_{\gamma,s}.
	\end{align*}

	Combining the estimates for all cases with the simple observation $co(S_{1},S_{2}) \subseteq B_{S_{1},S_{2}}$ 
yields the following lemma for the quadrature error.

	\begin{lemma}\label{consistency_error_lemma_multi}
	Let $\mathcal{T}_{\gamma}$ be a $\gamma$-shape regular mesh satisfying Assumption~\ref{assumption:finitely-many-patches}. Let $S_{1}, S_{2} \in \mathcal{T}_{\gamma}$ be closed simplices in $\mathbb{R}^{d}$ and denote by
	 $B_{S_{1},S_{2}}$ a closed ball with $\operatorname{diam} B_{S_1,S_2} \leq h_{S_1} + h_{S_2} + \operatorname{dist}(S_1,S_2)$ that contains the simplices $S_1,\, S_2 \subseteq B_{S_1,S_2}$. Then, for the integral $ I_{S_{1},S_{2}}(v,w) $ from \eqref{corollar_2d_eq1}
		and its approximation $ Q^{n}_{S_{1},S_{2}}(v,w) $ by quadrature, 
there exists a constant $\rho>1$ that depends only on $\gamma$ and $\Omega$ such that for all $v \in S_{0}^{p,1}(\mathcal{T}_{\gamma})$, $w \in S_{0}^{r,1}(\mathcal{T}_{\gamma})$ there holds
		\begin{align}
			|I_{S_{1},S_{2}}(v,w) - Q^{n}_{S_{1},S_{2}}(v,w)| \le C_{s,\gamma,d} \: (q_{\Phi}p)^{d+4} (q_{\Phi}r)^{d+4}  \rho^{q_\Phi(p+r)-2n+1} \big|   v   \big|_{ H^{s} ( B_{S_{1},S_{2}}) } \big|   w   \big|_{ H^{s} ( B_{S_{1},S_{2}}) }
		\end{align}
		with the constant $C_{s,\gamma,d}$ depending only on $s$, $\gamma$, $d$ and $\Omega$; $q_\Phi$ is given by Proposition~\ref{2d_proposition}.
	\end{lemma}

\subsection{Treatment of $\Omega^c$}
In this section, we discuss the issue that the evaluation of the bilinear form $a(\cdot,\cdot)$ requires the evaluation of 
$I_{S_{1},\Omega^c}$ given by (\ref{2D_int}).  This is addressed using two ingredients: 
\begin{enumerate}[nosep, label=(\roman*),leftmargin=*]
\item 
\label{item:omegac-i}
we select a set $B_R$ with $\overline{\Omega} \subset B_R$ (for convenience, 
this set will be taken to be a hypercube $[-R,R]^d$ below) and extend the mesh ${\mathcal T}_\gamma$ to a triangulation 
${\mathcal T}^R_\gamma$ of $B_R$ satisfying Assumption~\ref{assumption:finitely-many-patches}. For this triangulation, we may employ
the quadrature technique used above. 
\item 
\label{item:omegac-ii}
We develop a quadrature rule for integration over $B^c_R:= {\mathbb R}^d \setminus B_R$ and exploit that 
$\operatorname{dist}(T, B^c_R) \ge \operatorname{dist}(\Omega, B^c_R)  > 0$ together with analyticity of the integrand. 
\end{enumerate}
We focus on \ref{item:omegac-ii}. Let $B_R:=[-R,R]^d$ for a fixed $R > 0$. 
Introduce the cones ${\mathcal C}_1:= \{(y_1, y_1 y')\,|\, y_1 > R, y' \in [-1,1]^{d-1}\}$ as well as
${\mathcal C}_i$, $i=2,\ldots,2d$ obtained by rotating ${\mathcal C}_1$ so that the centerline of ${\mathcal C}_i$ is 
aligned with one of the unit vectors $(\pm 1,0,\ldots,0)$, $(0,\pm 1, 0,\ldots,0)$.  An integral of the kernel function over ${\mathcal C}_1$  
can be evaluated using the transformation $\eta = 1/y_1$ as follows: 
\begin{align*}
G_1(\vec{x})& := 
\int_{\vec{y} \in {\mathcal C}_1} |\vec{x} - \vec{y}|^{-(d+2s)}\, d\vec{y} 
= \int_{y' \in  [-1,1]^{d-1}} \int_{y_1=R}^\infty  |\vec{x} - y_1 (1,y')^\top|^{-(d+2s)} y_1^{d-1} \, dy'\, dy_1 \\
 &= \int_{y' \in  [-1,1]^{d-1}} \int_{\eta=0}^{1/R} \underbrace{| \eta \vec{x} -(1,y')^\top|^{-(d+2s)}}_{=:\mathcal{G}_1(\vec{x},\eta,y')} \eta^{2s-1} \, d\eta \, dy'. 
\end{align*}
This suggests to use a tensor product quadrature with (product) Gauss-Legendre quadrature in the $y'$-variables and a Gauss-Jacobi quadrature with weight $\eta^{2s-1}$
in the $\eta$-variable. 
Key to the performance of the quadrature rule is the analyticity of the function ${\mathcal G}_1$: 
\begin{lemma}
\label{lemma:omegac-analyticity}
Let $\overline{\Omega} \subset B_R$. Then:
\begin{enumerate}[nosep, label=(\roman*),leftmargin=*]
\item
\label{item:lemma:omegac-analyticity-i}
The function $G_1$ is analytic on $\overline{\Omega}$. 
\item
\label{item:lemma:omegac-analyticity-ii}
The function ${\mathcal G}_1$ is analytic on $\overline{\Omega}  \times [0,1/R] \times [-1,1]^{d-1}$. 
\item 
\label{item:lemma:omegac-analyticity-iii}
The functions $G_1$ and ${\mathcal G}_1$ are positive on $\overline{\Omega}$ and 
$\overline{\Omega}  \times [0,1/R] \times [-1,1]^{d-1}$, respectively. 
\end{enumerate}
Analyticity of a function $G$ on a closed set $A\subset {\mathbb R}^n$ means that there is a complex neighborhood $A_\varepsilon \subset {\mathbb C}^n$
of $A$ and a function $G_\varepsilon$ holomorphic on $A_\varepsilon$ with $G_\varepsilon|_A = G$. 
\end{lemma}
\begin{proof}
\emph{Proof of \ref{item:lemma:omegac-analyticity-ii}:} Consider the function 
$\widehat G(x_1,\ldots,x_d,\eta,y^\prime_2,\ldots,y^\prime_d):= (\eta x_1 - 1)^2 + \sum_{i=2}^d (\eta x_i - y^\prime_i)^2$, 
which is an entire function on ${\mathbb C}^{2d}$. We claim that $\widehat{G}(\vec{x},\eta,y^\prime) > 0$ 
on $K:= \overline{\Omega}  \times [0,1/R] \times [-1,1]^{d-1}$. By smoothness of $\widehat G$ and compactness of the set $K$
it suffices to show pointwise positivity of $\widehat{G}$. 
By construction of $\widehat{G}$, we have $\widehat{G} \ge (\eta \operatorname{dist}(\Omega,B^c_R))^2 > 0$ for $\eta > 0$. 
For $\eta = 0$, we have $\widehat{G}(\vec{x},0,y^\prime) = |(1,y^\prime)|^2 \ge 1$. 
Next, by positivity of $\widehat{G}$ on $K$ and the smoothness of $\widehat{G}$, there is a complex neighborhood 
$K_\varepsilon:= \cup_{z \in K} B_\varepsilon(z) \subset {\mathbb C}^{2d}$ such that $\operatorname{Re} \widehat G > 0$ on $D_\varepsilon$. 
Hence, with the principal branch of the logarithm, the function $\exp(-\frac{d+2s}{2} \log \widehat G(z))$ is holomorphic 
on $K_\varepsilon$ and coincides with $\mathcal{G}_1$ on $K$. 

\emph{Proof of \ref{item:lemma:omegac-analyticity-i}:} This follows from \ref{item:lemma:omegac-analyticity-ii}.
\end{proof}
In total, we have arrived at 
\begin{align*}
I_{S_{1},\Omega^c}(v,w) & = \sum_{S_{2} \in {\mathcal T}^R_\gamma \setminus {\mathcal T}_\gamma} I_{S_{1},S_{2}}(v,w) + \sum_{i=1}^{2d} \int_{\vec{x} \in T} G_i(\vec{x}) v(\vec{x}) w(\vec{x})\,d\vec{x}, 
\end{align*}
where the functions $G_i$, $i \ge 2$, are defined as $G_1$ with ${\mathcal C}_1$ replaced with ${\mathcal C}_i$. Analogous to 
Lemma~\ref{lemma:omegac-analyticity}, the functions $G_i$ and the corresponding integrands ${\mathcal G}_i$ are analytic. For a fully discrete approximation
of $I_{S_{1},\Omega^c}(v,w)$, we denote by $Q^n_{S_{1},{\mathcal C}_1}(v,w)$ the quadrature rule to evaluate 
$$
\int_{\vec{x} \in S_{1}} v(\vec{x}) w(\vec{x}) \int_{y^\prime \in [-1,1]^{d-1}} \int_{\eta=0}^{1 /R} {\mathcal G}_1(\vec{x},\eta,y^\prime) \eta^{2s-1} \,d\eta dy^\prime d\vec{x}
$$
with a tensor product Gauss-Legendre rule (with $n$ points for each variable) for the integration in $y^\prime$, a Gauss-Jacobi rule (with $n$ points) for the integration in 
$\eta$, and the tensor product Gauss-Legendre rule \eqref{Gauss_Simplex} for the integration in $\vec{x}$ over the simplex $S_{1}$. Analogously, we define rules $Q^n_{S_{1},{\mathcal C}_i}$, $i \ge 2$. The fully discrete approximation
is then given by
\begin{align}
\label{eq:Qnomegac}
I_{S_{1},\Omega^c}(v,w) & \approx Q^n_{S_{1},\Omega^c}(v,w):= \sum_{S_{2} \in {\mathcal T}^R_\gamma \setminus {\mathcal T}_\gamma} Q^n_{S_{1},S_{2}}(v,w) 
+ \sum_{i=1}^{2d} Q^n_{S_{1},{\mathcal C}_i} (v,w). 
\end{align} 

\begin{remark}
The function $\vec{x} \mapsto \int_{B^c_R} |\vec{x} - \vec{y}|^{-(d+2s)}\, d\vec{y}$ is analytic on $\overline{\Omega}$. Hence, 
it could be approximated by a (piecewise) polynomial on a coarse mesh. A computational speed-up is then possible since the evaluation of the 
$Q^n_{S_{1},{\mathcal C}_i}(v,w)$ can be replaced with the evaluation of $\int_{S_{1}} v(\vec{x}) w(\vec{x}) \pi(\vec{x})\,d\vec{x}$ for some polynomials $\pi$. Precomputing 
on the reference element is an option.  
\eremk
\end{remark}


\subsection{Exponential convergence under quadrature} 

Combining the approximation results for the integrals $I_{S_{1},S_{2}}(v,w)$ and $I_{S_{1},\Omega^c}(v,w)$ from the previous subsections, we directly arrive at an error estimate for the consistency error for the bilinear form $a(\cdot,\cdot)$.
	
	\begin{lemma}[Consistency error for $a$ for $d>1$] \label{consistency_error_lemma_bilinear_multid}
		Let  $\mathcal{T}_{\gamma}$ be a $\gamma$-shape regular mesh of $\Omega \subset \mathbb{R}^{d}$, $R>0$ be such that $\Omega \subset [-R,R]^d$ and $\mathcal{T}_{\gamma}^R$ be a $\gamma$-shape regular mesh that extends the mesh $\mathcal{T}_{\gamma}$ to $[-R,R]^d$. 
Assume ${\mathcal T}^R_\gamma$ satisfies Assumption~\ref{assumption:finitely-many-patches}.
Let $ a(\cdot,\cdot) $ be the bilinear form of (\ref{weak_formulationMulti}) and $ \widetilde{a}_{n}(\cdot,\cdot) $ be its approximation given by \eqref{bilinear_aprox_definition_mehrd}. Then, there exists a constant $\rho > 1$ that depends only on the shape regularity constant $\gamma$  and $\Omega$ such that for all $v \in S_{0}^{p,1}(\mathcal{T}_{\gamma}) $ and $w \in S_{0}^{r,1}(\mathcal{T}_{\gamma})$ there holds 
		\begin{align}
			|a(v,w)-\widetilde{a}_{n}(v,w)| \le C_{s,\gamma,d}(\# \mathcal{T}^R_{\gamma})^{2}p^{d+4} r^{d+4}  \rho^{q_\Phi(r+p)-2n+1} \|v\|_{\widetilde{H}^{s}(\Omega)} \|w\|_{\widetilde{H}^{s}(\Omega)},
		\end{align}
with constants $C_{s,\gamma,d}$ depending only on $s,\gamma,d$ and $\Omega$; $q_\Phi$ is given by Proposition~\ref{2d_proposition}. 
	\end{lemma}

\begin{proof}
By definition of $a_n(\cdot,\cdot)$, we have to distinguish the cases of  double integrals over simplices $I_{S_1,S_2}$ and integrals involving the complement $I_{S_{1},\Omega^c}$ and their approximation. The first case  can be done in the same way as for $d=1$ in \eqref{proof_of_consistency_error_lemma_bilinear}. 

For the integrals $I_{S_{1},\Omega^c}$ and their approximations $Q^n_{S_{1},\Omega^c}$, we mention that the contribution $Q^n_{S_{1},S_{2}}$ with $S_{2} \in {\mathcal T}^R_\gamma \setminus {\mathcal T}_\gamma$ can be treated as in the first case, replacing only the term $\# \mathcal{T}_{\gamma}$ by $\# \mathcal{T}^R_{\gamma}$. The other contributions of the form $Q^n_{S_{1},{\mathcal C}_i} (v,w)$ correspond to approximation of $\int_{S_{1}} v(\vec{x}) w(\vec{x}) G_i(\vec{x})\, d\vec{x}$ with analytic functions $G_i$ and thus take the same form as the integrals involved in the linear form $l(\cdot)$. Thus, a combination of Lemma~\ref{consistency_error_lemma_multi} and Lemma~\ref{consistency_error_lemma_functional_multi} together with summation over all simplices gives the result.
\end{proof}

	With the estimate for the consistency error, we directly obtain uniform coercivity as in the one dimensional case by a perturbation argument as described 
in Lemma~\ref{lemma:uniform_coercivity}. Note that the integral transformations for $d>1$ induce an additional constant $q_\Phi$ in the exponential term in the consistency error. In order to compensate for that the number of quadrature points now has to grow like $\lambda p$ for some $\lambda>1$.

	\begin{theorem}[Uniform coercivity, $d>1$]
		\label{th:uniform_coercivity_mehrd}
		Let the assumptions of Lemma~\ref{consistency_error_lemma_bilinear_multid} hold. Then, there are constant  $\widetilde{\alpha}, \lambda_1, \lambda_2 >0$  
		depending only on $s$, the shape regularity constant $\gamma$, the dimension $d$, and $\Omega$
		such that for $n\ge \lambda_1 p + \lambda_2 \ln(\# \mathcal{T}^R_{\gamma}+1)$ there holds 
		\begin{align}
			\widetilde{\alpha} \|v\|^{2}_{\widetilde{H}^{s}(\Omega)} \le \widetilde{a}_{n}(v,v) ~~~~\text{for all } v \in S^{p,1}_0(\mathcal{T}_{\gamma}).
		\end{align}
	\end{theorem}
		
	Now, employing the Strang Lemma, we can derive a result similar to Theorem~\ref{main_result} for $d>1$ by the exact same arguments. The error of the fully discrete FEM approximation can be bounded by the exact FEM error and a consistency error that decays exponentially in the number of quadrature points. 
	\begin{theorem}[exponential convergence under quadrature, $d>1$]
		\label{theorem_approx_galerkin_mehrd}
		Let $\mathcal{T}_{\gamma}$ be a $\gamma$-shape regular mesh of the bounded polyhedron $ \Omega \subset \mathbb{R}^{d}$. 
Let $R>0$ be such that $\Omega \subset [-R,R]^d$ and $\mathcal{T}_{\gamma}^R$ be a $\gamma$-shape regular mesh that extends the mesh $\mathcal{T}_{\gamma}$ to $[-R,R]^d$. 
Assume ${\mathcal T}^R_\gamma$ satisfies Assumption~\ref{assumption:finitely-many-patches}.
Let $f$ be analytic in $\overline{\Omega}$. Denote by $u \in \widetilde{H}^s(\Omega)$ the solution to \eqref{weak_formulationMulti}, by $u_{r} \in S_{0}^{r,1}(\mathcal{T}_{\gamma})$ the FEM solution for the exact variational formulation in the space $ S_{0}^{r,1}(\mathcal{T}_{\gamma}) \subseteq S_{0}^{p,1}(\mathcal{T}_{\gamma})$  and by $\widetilde{u}_{N,n} \in S_{0}^{p,1}(\mathcal{T}_{\gamma})$ the solution to
$$
\widetilde a_n(\widetilde{u}_{N,n},v_N) = \widetilde{l}(v_N) \qquad \forall v_N \in S_{0}^{p,1}(\mathcal{T}_{\gamma}),
$$
where $ \widetilde{a}_{n}(\cdot,\cdot) $ and $ \widetilde{l}_{n}(\cdot) $ are defined in (\ref{bilinear_aprox_definition_mehrd}) and (\ref{functional_aprox_definition_multi}). The index $n$ indicates the number of quadrature points that is used per coordinate direction per integral and element.
		
		Then, there exist constants $\rho>1,$ $\lambda_{1}$, $\lambda_{2}$, $C_{s,\gamma,d}>0$ (depending only on $s$, $\Omega$, $d$, $\gamma$), 
such that for all $p$, $\# \mathcal{T}_{\gamma}^R$ and $n$ with $n\ge \lambda_1 p+ \lambda_2\ln(\# \mathcal{T}_{\gamma}^R+1))$ and $r \in \mathbb{N} $ with $1 \le r \le p$ there holds
		\begin{align}\label{ungl_main_result_mehrd}
			\|u-\widetilde{u}_{N,n}\|_{\widetilde{H}^{s}(\Omega)} &\le \|u-u_{r}\|_{\widetilde{H}^{s}(\Omega)} + C_{s,\gamma,d} (\# \mathcal{T}^R_{\gamma})^{2} p^{d+4} r^{d+4} \rho^{q_\Phi(p+r)-2n+1};
		\end{align}
the constant is $q_\Phi$ given by Proposition~\ref{2d_proposition}. The number of operations to compute the stiffness is 
${\mathcal{O}}((n p) ^{2d} (\#{\mathcal T}^R_\gamma)^2)$. 
	\end{theorem}

\begin{remark}
The treatment of the complementary part $\Omega^c$ in the bilinear form induces the appearance of the term $\# \mathcal{T}^R_{\gamma}$ in the error estimate 
(\ref{ungl_main_result_mehrd}). 
In the context of shape-regular $hp$-FEM a natural choice for our model problem are ``boundary concentrated meshes'' both for $\mathcal{T}_\gamma$ and $\mathcal{T}^R_{\gamma}$ that are refined towards $\partial\Omega$ as discussed in \cite{khoromskij-melenk03}. 
The total number of elements is then proportional to the number of elements touching the boundary $\partial \Omega$ and thus $\# \mathcal{T}^R_{\gamma}$ is proportional to $\# \mathcal{T}_{\gamma}$.
\eremk
\end{remark}
	
\section{Numerical experiments} \label{ch:numerics}
In this section, we present some numerical examples that underline the theoretical estimates in our main results, Theorem~\ref{main_result}.
We consider 
\begin{align*}
(-\Delta)^s u = 1 \quad \text{ in } \Omega := (-1,1)  \text{ and } u = 0 ~ \text{ on } \Omega^c,
\end{align*}
with exact solution $u(x) = 2^{-2s}\sqrt{\pi}(\Gamma(s+1/2) \Gamma(1+s))^{-1} (1-x^{2})^{s}$.
\bigskip

In the following, we will present three different approaches to estimate the energy norm error 
between the exact solution $u$ and the fully discrete $hp$-FEM approximation $\widetilde{u}_{N,n}$
$$
\sqrt{a(u-\widetilde{u}_{N,n}, u-\widetilde{u}_{N,n})} = \sqrt{a(u,u)-a(\widetilde{u}_{N,n}, \widetilde{u}_{N,n})-2a(u-\widetilde{u}_{N,n}, \widetilde{u}_{N,n}}). 
$$ 
If the quadrature error is ignored, i.e., if it is assumed that $u_{N} = \widetilde{u}_{N,n}$, then Galerkin orthogonality 
$a(u - \widetilde{u}_{N,n}, \widetilde{u}_{N,n}) = 0$ \revision{holds} and, assuming that $a(u,u)$ is known, the error can be computed as the square root of the difference of the energies. The exact energy $a(\widetilde{u}_{N,n}, \widetilde{u}_{N,n})$ of $\widetilde{u}_{N,n}$ can in general 
only be approximated by quadrature, leading to an error estimate of the form
\begin{align}\label{energy_norm_approx}
	\sqrt{a(u-\widetilde{u}_{N,n}, u-\widetilde{u}_{N,n})} \approx \sqrt{a(u,u)-\widetilde{a}_{m}(\widetilde{u}_{N,n}, \widetilde{u}_{N,n})},
\end{align}
where $m \ge n$ denotes a number of quadrature points used. 

However, \revision{ for $\widetilde{u}_{N,n}$ the Galerkin orthogonality holds only up to the consistency error} as $u_{N}$ and $\widetilde{u}_{N,n}$ solve different variational formulations. 
For a high number of quadrature points $n$ the consistency error is small in comparison with the approximation error. However, 
for $n$ close to the polynomial degree $p$ we need a different approach.  The idea is to calculate an additional reference solution $\widetilde{u}_{N,m}$ with an increased number of quadrature points $m \gg n$ and use the triangle inequality to estimate the energy norm error by
\begin{align}\label{energy_norm_approx_triangle}
	\sqrt{a(u-\widetilde{u}_{N,n}, u-\widetilde{u}_{N,n})} \le \sqrt{a(u-\widetilde{u}_{N,m}, u-\widetilde{u}_{N,m})} + \sqrt{a(\widetilde{u}_{N,m}-\widetilde{u}_{N,n}, \widetilde{u}_{N,m}-\widetilde{u}_{N,n})}.
\end{align} 
By choosing $m$ sufficiently large, we can again use approximation (\ref{energy_norm_approx}) for the first term of the right hand-side. The second term can be approximated with the same small consistency error
\begin{align}
\label{method-3}
	\sqrt{a(\widetilde{u}_{N,m}-\widetilde{u}_{N,n}, \widetilde{u}_{N,m}-\widetilde{u}_{N,n})} \approx \sqrt{\widetilde{a}_{m}(\widetilde{u}_{N,m}-\widetilde{u}_{N,n}, \widetilde{u}_{N,m}-\widetilde{u}_{N,n})}.
\end{align}
We can interpret the first term in (\ref{energy_norm_approx_triangle}) as a good approximation of the energy norm error $\sqrt{a(u-u_{N}, u-u_{N})} $ 
and the second term in (\ref{energy_norm_approx_triangle}) as the implementation error. The following example shows that the difference between the approximation methods (\ref{energy_norm_approx}) and (\ref{energy_norm_approx_triangle}) can be significant.

\begin{figure}[h]
	\centering
\includegraphics{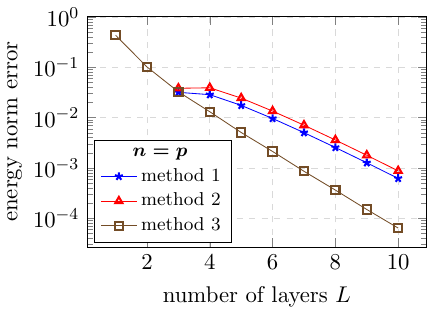} \quad
\includegraphics{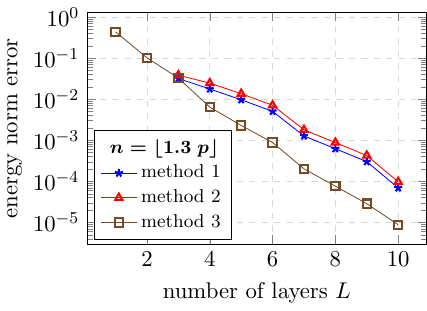} 
\includegraphics{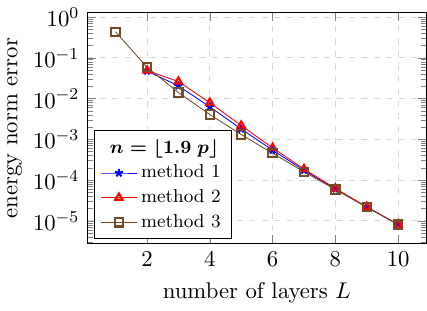}   \quad 
\includegraphics{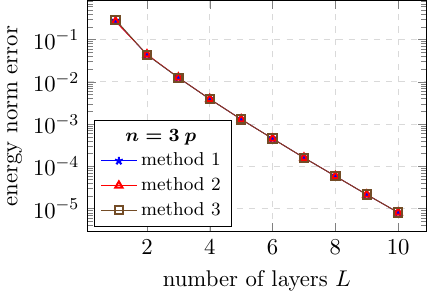} 
	\caption{Three different methods (see Example~\ref{example_energy_error_approx}) to calculate the energy norm error of $hp$-FEM with $n = \mathcal{O}(p)$ quadrature points on a geometric mesh with grading factor $\sigma = 0.172$, polynomial degree $p = L$, $s = 3/4$.}
	\label{fig:plot_error_approx}%
\end{figure}
\begin{example}\label{example_energy_error_approx}
	We employ a geometric mesh $\mathcal{T}^{L}_{geo,\sigma}$  with grading factor $\sigma = 0.172$ and take piecewise polynomials of degree $p = L$. In Figure \ref{fig:plot_error_approx}, three different error measures are plotted versus the number of refinement layers $L$ for different numbers of quadrature points $n = \mathcal{O}(p)$ used to calculate the solution $\widetilde{u}_{N,n}$:
\begin{itemize}
\item Method 1: Use approximation (\ref{energy_norm_approx}) with the same number of quadrature points $m$ for $\widetilde{a}_{m}(\cdot,\cdot)$ as for the solution $\widetilde{u}_{N,n}$, i.e. $m=n$.

\item Method 2: Use approximation (\ref{energy_norm_approx}) and increase the number of quadrature points for the bilinear form $\widetilde{a}_{m}(\cdot,\cdot) $ to $m = 6p$.

\item Method 3: 
Use approximation (\ref{energy_norm_approx_triangle}) with $m = 6p$ quadrature points for the reference solution $\widetilde{u}_{N,m}$ and the bilinear form $\widetilde{a}_{m}(\cdot,\cdot) $.
\end{itemize}
For the cases $n = \lfloor1.9 \; p \rfloor$ and $ n = 3\: p$ all three methods produce nearly identical results, whereas for $n = p $ and $n = \lfloor1.3 \; p \rfloor$ the method of calculating the norm has a significant impact. We observe that method 1 overestimates the energy norm error significantly and also increasing the number of quadrature points for the norm calculation (method 2) does not help either. This is consistent with the fact that method 2 does not decrease the consistency error that is made in the Galerkin orthogonality. We also note that for the cases $n = p$ and $n = \lfloor 1.3 p\rfloor$ the computed ``energies'' were larger than the exact energy so that 
no errors are reported for these cases in Fig.~\ref{fig:plot_error_approx}. 

\end{example}

The next example is similar to an example in \cite{bahr2023exponential} that shows exponential convergence of $hp$-FEM, where the linear system was assembled using the quadrature approach (\ref{eq:fully_discrete}) in this article.

\begin{example} \label{example_conv_rate}
We employ a geometric mesh $\mathcal{T}^{L}_{geo,\sigma}$  with grading factor $\sigma = 0.25$ and take piecewise polynomials of degree $p = L$. In Figure~\ref{fig:plot1}, the energy norm error (approximation (\ref{energy_norm_approx_triangle}) with $m = 6p$) is plotted versus the number of refinement layers $L$ for different fractional parameters $s$. For the number of quadrature points, we used $n:= \lfloor1.2 \; p \rfloor$ and, as predicted by Theorem~\ref{main_result}, we observe exponential convergence with respect to the number of layers $L$ noting that $N \sim L^2$. In fact, the convergence behavior is $\mathcal{O}(\sigma^{L/2}L^{-1})$ and thus slightly faster than \revision{asserted} by 
Theorem~\ref{main_result}. An argument for this observation is given in \cite[Sec.~{4}]{bahr2023exponential}. 

\begin{figure}[h]
	\centering
\includegraphics{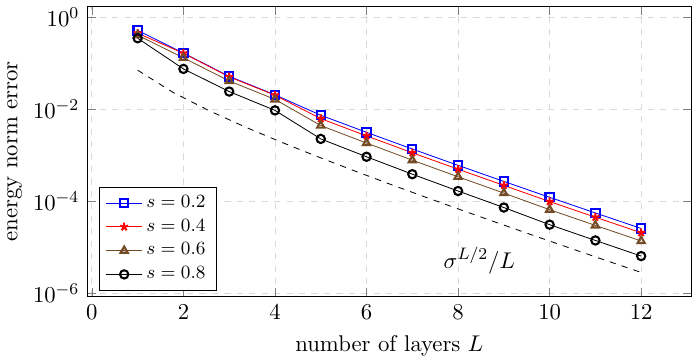}
	\caption{Exponential convergence in the energy norm (approximation (\ref{energy_norm_approx_triangle}) with $m = 6p$) of $hp$-FEM on geometric mesh with grading factor $\sigma = 0.25$, polynomial degree $p = L$, $n:= \lfloor1.2 \; p \rfloor$ quadrature points and different fractional parameters $s$.}
	\label{fig:plot1}
\end{figure}
\end{example}

Next, we discuss the number of quadrature points used. Although Theorem~\ref{main_result} suggests that \revision{a choice of quadrature points $n\ge p + 1$} and in particular $n:= p + \widetilde{\lambda} \; p$ for all $\widetilde{\lambda}>0,$ suffices to obtain exponential convergence, the rate, or more precisely, the constant in the exponent, is impacted by the choice of $\widetilde{\lambda}$. 

\begin{example}
Figure~\ref{fig:plot2} plots the energy norm error (approximation (\ref{energy_norm_approx_triangle}) with $m = 6p$) for different numbers of quadrature points $n:= p + \widetilde{\lambda} p$ versus the number of layers $L$ for two different choices of grading parameters, $\sigma = 0.172$ and $\sigma = 0.5$. Again, we choose $p = L$ and fix the fractional parameter $ s = 3/4$. We notice that the grading factor $\sigma$ has a direct impact on the number of quadrature points needed to achieve the same accuracy. For the smaller $\sigma = 0.172$, the rate of the exponential convergence depends on the choice of $\widetilde{\lambda} $, while, for $\sigma = 0.5$, the convergence always appears to be $\mathcal{O}(\sigma^{L/2}L^{-1})$. This can also be observed in the theoretical estimates in Theorem \ref{main_result} as the term $L^{2} p^{6} \rho^{1+2p-2n}$  may be dominant in the case of small $\sigma$.

\begin{figure}[h]
	\centering
\includegraphics{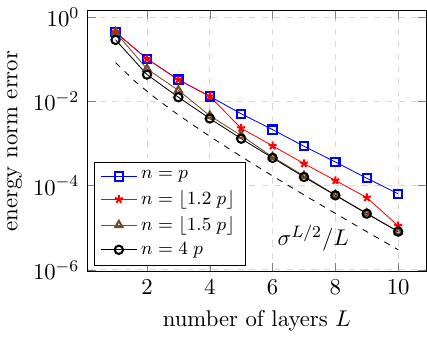} \quad
\includegraphics{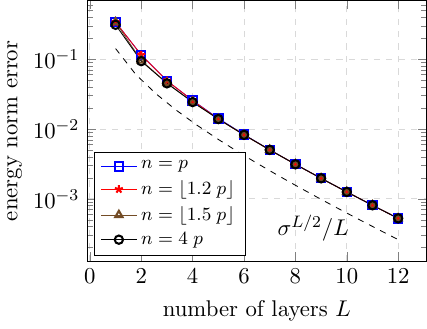} 
	\caption{Exponential convergence in the energy norm (approximation (\ref{energy_norm_approx_triangle}) with $m = 6p$) of $hp$-FEM with $n = \mathcal{O}(p)$ quadrature points on geometric mesh, polynomial degree $p = L$, $s = 3/4$. Left: grading factor $\sigma = 0.172$. Right: grading factor $\sigma = 0.5$.}
	\label{fig:plot2}%
\end{figure}
\end{example}

Finally, we consider the elementwise contributions in Lemma~\ref{lemma_quadrature_error} and observe exponential convergence for two different configurations.

\begin{example}
Figure~\ref{fig:plot3} considers the case of adjacent elements $T:=(x^{geo}_0,x^{geo}_1),\: T':=(x^{geo}_1,x^{geo}_2)$ (left) 
and separated elements $T:=(x^{geo}_0, x^{geo}_1),T':=(x^{geo}_2, x^{geo}_3)$ (right) in a geometric mesh $\mathcal{T}^{L}_{geo,\sigma}$ 
with $L=2$ layers and different grading parameters $\sigma$ (see Def.~\ref{geometric_mesh_definition}).
We plot the absolute quadrature errors $ | I_{T,T'}(v,w) - Q^n_{T,T'}(v,w) |$ for two integrated Legendre polynomials $v:T\rightarrow \mathbb{R}$ and $w:T'\rightarrow \mathbb{R}$  versus the number of quadrature points $n$. On the reference domain $(-1,1)$ they are defined as 
\begin{align}\label{def_int_legendre_pol}
	v(x)=\int_{-1}^{x}P_{5}(t) dt~~ \text{ and }~~ w(y)=\int_{-1}^{y}P_{7}(t) dt,
\end{align}
where $P_{i}(t) \in \mathcal{P}_{i}$ denotes the $i$-th Legendre polynomial.
We used $Q^{50}_{T,T'}(v,w)$ with 50 quadrature points, as the reference solution $ I_{T,T'}(v,w) $ and observe the predicted exponential convergence rate as well as that the rate decreases with  $\sigma$. This is in line with Lemma~\ref{lemma_quadrature_error} since $\rho \rightarrow 1 $ as $\sigma \rightarrow 0$. We stress that Figure~\ref{fig:plot3} shows the absolute error; the final relative error is close to machine precision.

\begin{figure}[h]
	\centering
\includegraphics{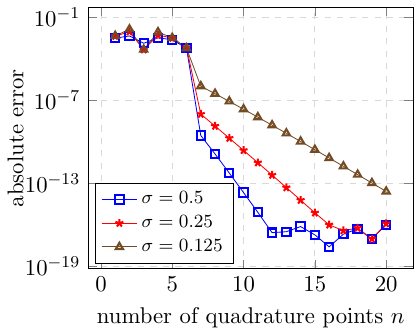} \quad
\includegraphics{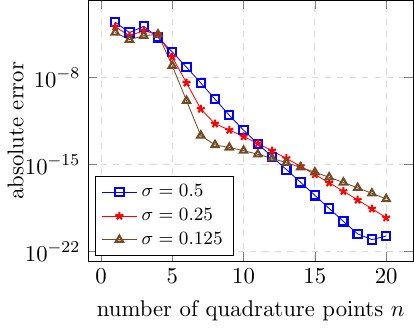} 
	\caption{Exponential convergence of the elementwise contributions $| I_{T,T'}(v,w) - Q^n_{T,T'}(v,w) |$ for the integrated Legendre polynomials (\ref{def_int_legendre_pol}) on geometric meshes $\mathcal{T}^{L}_{geo,\sigma}$ with $L=2$ layers and different grading parameters $\sigma$. Left: adjacent elements. Right: separated elements.}
	\label{fig:plot3}%
\end{figure}
\end{example}

\subsection*{Acknowledgments}
BB and JMM gladly acknowledge financial support by Austrian Science Fund (FWF) 
throughout the special research program \textit{Taming complexity in PDE systems} (grant SFB F65,
\href{https://doi.org/10.55776/F65}{DOI:10.55776/F65}).


\bibliographystyle{alpha}
\bibliography{literature.bib}

\newcommand{\etalchar}[1]{$^{#1}$}
\begin{thebibliography}{DDG{\etalchar{+}}20}

\bibitem[AB17]{acosta2017femfractional}
G.~Acosta and J.P. Borthagaray.
\newblock A fractional {L}aplace equation: regularity of solutions and finite
  element approximations.
\newblock {\em SIAM J. Numer. Anal.}, 55(2):472--495, 2017.

\bibitem[ABB17]{acosta-bersetche-borthagaray17}
G.~Acosta, F.M. Bersetche, and J.P. Borthagaray.
\newblock A short {FE} implementation for a 2d homogeneous {D}irichlet problem
  of a fractional {L}aplacian.
\newblock {\em Comput. Math. Appl.}, 74(4):784--816, 2017.

\bibitem[AG18]{ainsworth-glusa18}
M.~Ainsworth and C.~Glusa.
\newblock Towards an efficient finite element method for the integral
  fractional {L}aplacian on polygonal domains.
\newblock In {\em Contemporary computational mathematics---a celebration of the
  80th birthday of {I}an {S}loan. {V}ol. 1, 2}, pages 17--57. Springer, Cham,
  2018.

\bibitem[BBN{\etalchar{+}}18]{bonito2018overview}
A.~Bonito, J.P. Borthagaray, R.H. Nochetto, E.~Ot\'{a}rola, and A.J. Salgado.
\newblock Numerical methods for fractional diffusion.
\newblock {\em Comput. Vis. Sci.}, 19(5-6):19--46, 2018.

\bibitem[BFM{\etalchar{+}}23]{bahr2023exponential}
B.~Bahr, M.~Faustmann, C.~Marcati, J.M. Melenk, and C.~Schwab.
\newblock {Exponential convergence of hp-FEM for the integral fractional
  Laplacian in 1D}.
\newblock In {\em Spectral and High Order Methods for Partial Differential
  Equations ICOSAHOM 2020+ 1: Selected Papers from the ICOSAHOM Conference,
  Vienna, Austria, July 12-16, 2021}, pages 291--306. Springer, 2023.

\bibitem[BLN20]{borthagaray2019overview}
J.P. Borthagaray, W.~Li, and R.H. Nochetto.
\newblock Linear and nonlinear fractional elliptic problems.
\newblock In {\em 75 years of mathematics of computation}, volume 754 of {\em
  Contemp. Math.}, pages 69--92. Amer. Math. Soc., Providence, RI, 2020.

\bibitem[Bra07]{braess2007finite}
D.~Braess.
\newblock {\em Finite elements: Theory, fast solvers, and applications in
  elasticity theory}.
\newblock Cambridge University Press, Cambridge, third edition, 2007.

\bibitem[BV16]{bucur2016applications}
C.~Bucur and E.~Valdinoci.
\newblock {\em Nonlocal diffusion and applications}, volume~20 of {\em Lecture
  Notes of the Unione Matematica Italiana}.
\newblock Springer, [Cham]; Unione Matematica Italiana, Bologna, 2016.

\bibitem[CR13]{chernov-reinarz13}
A.~Chernov and A.~Reinarz.
\newblock Numerical quadrature for high-dimensional singular integrals over
  parallelotopes.
\newblock {\em Comput. Math. Appl.}, 66(7):1213--1231, 2013.

\bibitem[CS12]{chernov-schwab12}
A.~Chernov and C.~Schwab.
\newblock Exponential convergence of {G}auss-{J}acobi quadratures for singular
  integrals over simplices in arbitrary dimension.
\newblock {\em SIAM J. Numer. Anal.}, 50(3):1433--1455, 2012.

\bibitem[CSW16]{chen-shen-wang16}
S.~Chen, J.~Shen, and L.-L. Wang.
\newblock Generalized {J}acobi functions and their applications to fractional
  differential equations.
\newblock {\em Math. Comp.}, 85(300):1603--1638, 2016.

\bibitem[CvPS11]{chernov-vonpetersdorff-schwab11}
A.~Chernov, T.~von Petersdorff, and C.~Schwab.
\newblock Exponential convergence of {$hp$} quadrature for integral operators
  with {G}evrey kernels.
\newblock {\em ESAIM Math. Model. Numer. Anal.}, 45(3):387--422, 2011.

\bibitem[CvPS15]{chernov-vonpetersdorff-schwab15}
A.~Chernov, T.~von Petersdorff, and C.~Schwab.
\newblock Quadrature algorithms for high dimensional singular integrands on
  simplices.
\newblock {\em Numer. Algorithms}, 70(4):847--874, 2015.

\bibitem[DDG{\etalchar{+}}20]{delia20overview}
M.~D'Elia, Q.~Du, C.~Glusa, M.~Gunzburger, X.~Tian, and Z.~Zhou.
\newblock Numerical methods for nonlocal and fractional models.
\newblock {\em Acta Numer.}, 29:1--124, 2020.

\bibitem[DL93]{devore1993constructive}
R.A. DeVore and G.G. Lorentz.
\newblock {\em Constructive approximation}, volume 303.
\newblock Springer Science \& Business Media, 1993.

\bibitem[FB23]{bebendorf-feist23}
B.~Feist and M.~Bebendorf.
\newblock Fractional {L}aplacian--quadrature rules for singular double
  integrals in 3{D}.
\newblock {\em Comput. Methods Appl. Math.}, 23(3):623--645, 2023.

\bibitem[FMMS23]{faustmann2023hp}
M.~Faustmann, C.~Marcati, J.M. Melenk, and C.~Schwab.
\newblock Exponential convergence of {$hp$}-{FEM} for the integral fractional
  {L}aplacian in polygons.
\newblock {\em SIAM J. Numer. Anal.}, 61(6):2601--2622, 2023.

\bibitem[Heu14]{heuer2014equivalence}
N.~Heuer.
\newblock {On the equivalence of fractional-order Sobolev semi-norms}.
\newblock {\em Journal of Mathematical Analysis and Applications},
  417(2):505--518, 2014.

\bibitem[HMS96]{holm-maischak-stephan96}
H.~Holm, M.~Maischak, and E.~P. Stephan.
\newblock The {$hp$}-version of the boundary element method for {H}elmholtz
  screen problems.
\newblock {\em Computing}, 57(2):105--134, 1996.

\bibitem[JZ23]{jin-zhou23}
B.~Jin and Z.~Zhou.
\newblock {\em Numerical treatment and analysis of time-fractional evolution
  equations}, volume 214 of {\em Applied Mathematical Sciences}.
\newblock Springer, Cham, 2023.

\bibitem[Kar19]{handbook_vol_3}
G.E. Karniadakis, editor.
\newblock {\em Handbook of fractional calculus with applications. {V}ol. 3}.
\newblock De Gruyter, Berlin, 2019.
\newblock Numerical methods.

\bibitem[KM03]{khoromskij-melenk03}
B.~N. Khoromskij and J.~M. Melenk.
\newblock Boundary concentrated finite element methods.
\newblock {\em SIAM J. Numer. Anal.}, 41(1):1--36, 2003.

\bibitem[KPP13]{karkulik-pavlicek-praetorius13}
M.~Karkulik, D.~Pavlicek, and D.~Praetorius.
\newblock {On 2D newest vertex bisection: optimality of mesh-closure and
  $H^1$-stability of $L_2$-projection}.
\newblock {\em Constr. Approx.}, 38:213--234, 2013.

\bibitem[KS05]{karniadakis-sherwin05}
G.E. Karniadakis and S.J. Sherwin.
\newblock {\em Spectral/{$hp$} element methods for computational fluid
  dynamics}.
\newblock Numerical Mathematics and Scientific Computation. Oxford University
  Press, New York, second edition, 2005.

\bibitem[LPG{\etalchar{+}}20]{lischke2020overview}
A.~Lischke, G.~Pang, M.~Gulian, F.~Song, C.~Glusa, X.~Zheng, Z.~Mao, W.~Cai,
  M.M. Meerschaert, M.~Ainsworth, and G.E. Karniadakis.
\newblock What is the fractional {L}aplacian? {A} comparative review with new
  results.
\newblock {\em J. Comput. Phys.}, 404:109009, 62, 2020.

\bibitem[LZZ19]{lischke-handbook}
A.~Lischke, M.~Zayernouri, and Z.~Zhang.
\newblock Spectral and spectral element methods for fractional
  advection-diffusion-reaction equations.
\newblock In {\em Handbook of fractional calculus with applications. {V}ol. 3},
  pages 157--183. De Gruyter, Berlin, 2019.

\bibitem[Mai95]{maischak95}
M.~Maischak.
\newblock {\em Hp-Methoden f\"ur {R}andintegralgleichungen bei
  {3D}-{P}roblemen: {T}heorie und {I}mplementierung}.
\newblock PhD thesis, Universit\"at Hannover, 1995.

\bibitem[McL00]{mclean00}
W.~McLean.
\newblock {\em Strongly elliptic systems and boundary integral equations}.
\newblock Cambridge University Press, 2000.

\bibitem[MS98]{melenk1998hp}
J.M. Melenk and C.~Schwab.
\newblock {hp FEM for Reaction-Diffusion Equations I: Robust Exponential
  Convergence}.
\newblock {\em SIAM journal on numerical analysis}, 35(4):1520--1557, 1998.

\bibitem[MS18]{mao-shen18}
Z.~Mao and J.~Shen.
\newblock Spectral element method with geometric mesh for two-sided fractional
  differential equations.
\newblock {\em Adv. Comput. Math.}, 44(3):745--771, 2018.

\bibitem[Sch98]{schwab1998p}
C.~Schwab.
\newblock p-and hp-finite element methods: Theory and applications in solid and
  fluid mechanics.
\newblock {\em Oxford: Clarendon Press}, 1998.

\bibitem[SS11]{sauter2011book}
S.A. Sauter and C.~Schwab.
\newblock {\em Boundary element methods}, volume~39 of {\em Springer Series in
  Computational Mathematics}.
\newblock Springer-Verlag, Berlin, 2011.
\newblock Translated and expanded from the 2004 German original.

\bibitem[SS19]{shen-handbook}
J.~Shen and C.~Sheng.
\newblock Spectral methods for fractional differential equations using
  generalized {J}acobi functions.
\newblock In {\em Handbook of fractional calculus with applications. {V}ol. 3},
  pages 127--155. De Gruyter, Berlin, 2019.

\bibitem[Ste08]{stevenson08}
R.~Stevenson.
\newblock The completion of locally refined simplicial partitions created by
  bisection.
\newblock {\em Math. Comp.}, 77(261):227--241, 2008.

\bibitem[ZK13]{zayernouri-karniadakis13}
M.~Zayernouri and G.E. Karniadakis.
\newblock Fractional {S}turm-{L}iouville eigen-problems: theory and numerical
  approximation.
\newblock {\em J. Comput. Phys.}, 252:495--517, 2013.

\bibitem[ZWSK24]{zayernouri-wang-shen-karniadakis23}
M.~Zayernouri, L.-L. Wang, J.~Shen, and G.E. Karniadakis.
\newblock {\em Spectral and Spectral-Element Methods for Fractional Ordinary
  and Partial Diﬀerential Equations}.
\newblock Cambridge University Press, 2024.

\end{thebibliography}


\end{document}